\documentclass[a4paper]{amsart}

\usepackage{amssymb} 
\usepackage[all]{xy}
\usepackage{hyperref, aliascnt}
\usepackage{mathtools}

\numberwithin{equation}{section}

\newtheorem{lma}{Lemma}[section]

\newaliascnt{thmCt}{lma}
\newtheorem{thm}[thmCt]{Theorem}
\aliascntresetthe{thmCt}

\newaliascnt{corCt}{lma}
\newtheorem{cor}[corCt]{Corollary}
\aliascntresetthe{corCt}

\newaliascnt{prpCt}{lma} 
\newtheorem{prp}[prpCt]{Proposition}
\aliascntresetthe{prpCt}

\newcounter{theoremintro}
\newtheorem{thmintro}[theoremintro]{Theorem}

\theoremstyle{definition}

\newaliascnt{pgrCt}{lma}
\newtheorem{pgr}[pgrCt]{}
\aliascntresetthe{pgrCt}

\newaliascnt{dfnCt}{lma}
\newtheorem{dfn}[dfnCt]{Definition}
\aliascntresetthe{dfnCt}

\newaliascnt{rmkCt}{lma}
\newtheorem{rmk}[rmkCt]{Remark}
\aliascntresetthe{rmkCt}

\newaliascnt{rmksCt}{lma}

\aliascntresetthe{rmksCt}

\newaliascnt{qstCt}{lma}
\newtheorem{qst}[qstCt]{Question}
\aliascntresetthe{qstCt}

\newaliascnt{pbmCt}{lma}
\newtheorem{pbm}[pbmCt]{Problem}
\aliascntresetthe{pbmCt}

\newaliascnt{exaCt}{lma}
\newtheorem{exa}[exaCt]{Example}
\aliascntresetthe{exaCt}

\newaliascnt{exasCt}{lma}
\newtheorem{exas}[exasCt]{Examples}
\aliascntresetthe{exasCt}

\newaliascnt{ntnCt}{lma}
\newtheorem{ntn}[ntnCt]{Notation}
\aliascntresetthe{ntnCt}

\DeclareMathOperator{\core}{core}
\DeclareMathOperator{\Homeo}{Homeo}
\DeclareMathOperator{\Isom}{Isom}
\DeclareMathOperator{\dom}{dom}
\DeclareMathOperator{\ran}{ran}
\DeclareMathOperator{\asdim}{asdim}
\newcommand{\frakA}{\mathfrak{A}}
\newcommand{\RR}{\mathbb{R}}
\newcommand{\CC}{\mathbb{C}}
\newcommand{\NN}{\mathbb{N}}
\newcommand{\ZZ}{\mathbb{Z}}
\newcommand{\I}{\infty}
\newcommand{\calG}{\mathcal{G}}

\newcommand{\calGZero}{\mathcal{G}^{(0)}}

\newcommand{\andSep}{\,\,\,\text{ and }\,\,\,}

\newcommand{\ca}{$C^*$-algebra}
\newcommand{\casub}{$C^*$-subalgebra}

\newcommand{\K}{{\mathcal{K}}}
\newcommand{\Bdd}{{\mathcal{B}}}
\newcommand{\Onp}{{\mathcal{O}_n^p}}
\newcommand{\Ot}{{\mathcal{O}_2}}
\newcommand{\Otp}{{\mathcal{O}_2^p}}
\newcommand{\id}{{\mathrm{id}}}
\newcommand{\supp}{{\mathrm{supp}}}
\newcommand{\Aut}{{\mathrm{Aut}}}
\newcommand{\spat}{{\mathrm{sp}}}
\newcommand{\tensProj}{\widehat{\otimes}}
\newcommand{\tensMax}{\otimes_{\mathrm{max}}}
\newcommand{\tensSp}{\otimes_{\mathrm{sp}}}
\DeclareMathOperator{\Rep}{Rep}

\title[]{Rigidity results for L\texorpdfstring{$^{\mathrm{\MakeLowercase{p}}}$}{p}-operator algebras and applications}

\subjclass[2010]{Primary:
43A15, %$L^p$-spaces and other function spaces on groups, semigroups, etc.
47L10; %Algebras of operators on Banach spaces and other topological linear spaces
Secondary:
22A22, %	Topological groupoids (including differentiable and Lie groupoids)
43A65. %Representations of groups, semigroups, etc.
}

\keywords{$L^p$-space, Banach algebra, crossed product, tensor product, groupoid, Cuntz algebra}

\date{\today}

\thanks{The second named author was partially supported by a Postdoctoral Research Fellowship from the Humboldt Foundation and by the Swedish Research Council Grant 2021-04561.
The second and third named authors were partially supported by the Deutsche Forschungsgemeinschaft (DFG, German Research Foundation) under the SFB 878 (Groups, Geometry \& Actions) and under Germany's Excellence Strategy EXC 2044-390685587 (Mathematics M\"{u}nster: Dynamics-Geometry-Structure).
The third named author was partially supported by the Knut and Alice Wallenberg Foundation (KAW 2021.0140).
}

\author[Yemon Choi]{Yemon Choi}
\address{Yemon Choi
School of Mathematical Sciences, Lancaster University, Lancaster
LA1 4YF, United Kingdom.}
\email{y.choi1@lancaster.ac.uk}
  \urladdr{www.maths.lancs.ac.uk/~choiy1/}

\author[Eusebio Gardella]{Eusebio Gardella}
\address{Eusebio Gardella
Department of Mathematical Sciences, Chalmers University of
Technology and University of Gothenburg, Gothenburg SE-412 96, Sweden.}
\email{gardella@chalmers.se}
\urladdr{www.math.chalmers.se/~gardella}

\author{Hannes Thiel}
\address{Hannes~Thiel, 
Department of Mathematical Sciences, Chalmers University of Technology and University of
Gothenburg, Gothenburg SE-412 96, Sweden.}
\email{hannes.thiel@chalmers.se}
\urladdr{www.hannesthiel.org}

\begin{document}

%==========================================================================================
\begin{abstract}
For $p\in [1,\infty)$, we show that every unital $L^p$-operator algebra contains a unique maximal $C^*$-subalgebra, which is always abelian if $p\neq 2$.
Using this, we canonically associate to every unital $L^p$-operator algebra $A$ an \'etale groupoid $\calG_A$, which in many cases of interest is a complete invariant for $A$.
By identifying this groupoid for large classes of examples, we obtain a number of rigidity results that display a stark contrast with the case $p=2$; the most striking one being that of crossed products by topologically free actions.

Our rigidity results give answers to questions concerning the existence of isomorphisms between different algebras. 
Among others, we show that for the $L^p$-analog $\Otp$ of the Cuntz algebra, there is no isometric isomorphism between $\Otp$ and $\Otp\otimes^p\Otp$, when $p\neq 2$.
In particular, we deduce that there is no $L^p$-version of Kirchberg's absorption theorem, and that there is no $K$-theoretic classification of purely infinite simple amenable $L^p$-operator algebras for $p\neq 2$. 
Our methods also allow us to recover a folklore fact in the case of C*-algebras ($p=2$), namely that no isomorphism $\Ot\cong \Ot\otimes\Ot$ preserves the canonical Cartan subalgebras.
\end{abstract}

%==========================================================================================
\maketitle

\renewcommand*{\thetheoremintro}{\Alph{theoremintro}}

%==========================================================================================
\section{Introduction}

Given $p\in [1,\infty)$, we say that a Banach algebra is an \emph{$L^p$-operator algebra}
if it admits an isometric representation on an $L^p$-space.
The case $p=2$ has been intensively studied and a rich theory has been developed;
for $p\in [1,\infty)\setminus\{2\}$, new challenges arise, and much less is known.

Historically, one important strand arises from Herz's influential works in the 1970s on harmonic analysis on $L^p$-spaces, unifying results of previous authors for abelian or compact groups. Given a locally compact group $G$, Herz studied the Banach algebra $PF_p(G)\subseteq \Bdd(L^p(G))$
generated by the left regular representation, 
as well as the commutant $CV_p(G)\subseteq \Bdd(L^p(G))$ 
of the right translation operators.
For $p=2$, these are respectively the reduced group \ca\ and group von Neumann algebra of $G$.
Both 
$PF_p(G)$ and $CV_p(G)$ have attracted the 
attention of a number of people
during the last two decades \cite{Cow98PredConv, Run05ReprQSL, Cho15DirFin, DawSpr19ConvLp};
an overview of the classical results can be found in \cite{Der11ConvOps} and a more recent survey
can be found in \cite{Gar21ModernLp}.

Many basic tools available in the $C^*$-algebraic setting may fail to hold for operators
on general $L^p$-spaces, making their study very challenging.
Nevertheless, $L^p$-operator algebras have recently seen renewed interest, thanks to the infusion of ideas, examples and techniques from operator algebras, particularly in the works of Phillips \cite{Phi12arX:LpAnalogsCtz, Phi13arX:LpCrProd}. 
There, he introduced and studied the $L^p$-analogs $\Onp$ of the Cuntz algebras $\mathcal{O}_n$ from \cite{Cun77SimpleCAlgGenIsom} (which are the case $p=2$).
These Banach algebras behave in many ways very similarly to the corresponding \ca{s}:
among others, they are simple, purely infinite, amenable, and their $K$-theory
is independent of $p$.
However, the proofs for $p=2$ and $p\neq 2$ differ drastically for most of these.

The work of Phillips has motivated other authors to study $L^p$-analogs of well-studied families of \ca{s}, 
including group algebras \cite{Phi13arX:LpCrProd, GarThi15GpAlgLp, GarThi19ReprConvLq}; 
groupoid algebras \cite{GarLup17ReprGrpdLp}; 
crossed products by topological systems \cite{Phi13arX:LpCrProd}; 
AF-algebras \cite{PhiVio20ClassLpAF, GarLup16NonclassLpUHF}; 
Roe algebras \cite{ChuLi18RigidityLpRoe, BraVig20RoeAlgAsBAlg}; 
and graph algebras \cite{CorGui19LpOpAlgGraphs}. 
In these works, an $L^p$-operator algebra is obtained from combinatorial or dynamical data, and properties
of the underlying data (such as hereditary saturation of a graph, or minimality of an action) are related
to properties of the algebra (such as simplicity).
More recent works have approached the study of $L^p$-operator algebras in a more
abstract and systematic way \cite{GarThi15BAlgInvIsoLp, GarThi20ExtendingRepr, BlePhi19LpOpAlgApproxId1}, 
showing that there is an interesting theory waiting to be unveiled.

The present paper takes a further step in this direction, by studying the internal structure of $L^p$-operator
algebras and their abelian subalgebras, specifically for $p\neq 2$. Our first main result is as follows:

%==========================================================================================
\begin{thmintro}[See \autoref{thm:C*core}]
\label{thmintro:C*core}
Let $p\in [1,\infty)$, and let $A$ be a unital $L^p$-operator algebra.
Then there is a unique maximal unital $C^*$-subalgebra $\core(A)$ of $A$,
called the \emph{$C^*$-core} of $A$. If $p\neq 2$, then $\core(A)$ is abelian.
\end{thmintro}

%==========================================================================================
Theorem~\ref{thmintro:C*core} can be interpreted as follows:
while a given unital $L^p$-operator algebra ($p\neq 2$) in general has many non-isomorphic maximal abelian subalgebras,
it has a \emph{unique} one that is isometrically of the form $C(X)$.
In particular, any isometric isomorphism
must preserve the $C^*$-cores. This is a dramatic
difference with \ca s, where even 
Cartan subalgebras are not unique. On the downside, 
the $C^*$-core can sometimes be too small to be of 
any use (see \autoref{exa:GroupoidGrpAlgs}).
For an $L^p$-operator algebra obtained from either a combinatorial object or a dynamical system, 
the $C^*$-core can often be computed in terms of the underlying data (see \autoref{thm:GroupoidRigidity}). This is a very useful tool that allows us to retrieve information about the combinatorics/dynamics from the associated algebra, which is best seen in the case of topologically free actions:

%==========================================================================================
\begin{thmintro}[See \autoref{thm:RigidityDynSysts}]
\label{thmintro:CPrigidity}
Let $p\in [1,\infty)\setminus\{2\}$, let $G$ and $H$ be discrete groups, let $X$ and $Y$ be compact
Hausdorff spaces, and let $G\curvearrowright X$ and $H\curvearrowright Y$ be topologically free actions.
Then $G\curvearrowright X$ and $H\curvearrowright Y$ are continuously orbit equivalent if and only if
there is an isometric isomorphism $F^p_\lambda(G,X)\cong F^p_\lambda(H,Y)$.
\end{thmintro}

%==========================================================================================
In other words, for $p\neq 2$, the $L^p$-crossed product of a topologically free action \emph{remembers}
the continuous orbit equivalence class of the given action, and hence also the quasi-isometry class of the acting group.
Again, this shows how much more rigid the case $p\neq 2$ is in comparison with $p=2$. For the sake of comparison, other rigidity phenomena, this time in the context of coarse geometry and uniform Roe algebras, have 
been obtained in \cite{ChuLi18RigidityLpRoe, BraVig20RoeAlgAsBAlg}.

As a further application of our methods, we show that there is no $L^p$-analog of Elliott's isomorphism theorem $\Ot\otimes\Ot\cong \Ot$.
More concretely:

%==========================================================================================
\begin{thmintro}[See \autoref{prp:O2O2Noniso}]
\label{thmintro:TensProdCtzAlgs}
Let $p\in [1,\infty)\setminus\{2\}$, and let $m,n\in\NN$.
Then there is an isometric isomorphism
\[
\underbrace{\Otp\otimes^p\cdots\otimes^p\Otp}_{m}
\cong \underbrace{\Otp\otimes^p\cdots\otimes^p\Otp}_{n}
\]
if and only if $m=n$.
\end{thmintro}
% 
%==========================================================================================
As a consequence, we answer a question of Phillips: there is no isometric isomorphism between
$\Otp$ and $\Otp\otimes^p\Otp$ for $p\neq 2$, although they are both
simple, purely infinite, amenable $L^p$-operator algebras with identical $K$-theory (see \autoref{prop:KtheoryTheSame}).
In particular, $K$-theory is not a fine enough invariant to distinguish
between simple, purely infinite, amenable
$L^p$-operator algebras, when $p\neq 2$, in contrast to the celebrated Kirchberg--Phillips classification of simple,
purely infinite, amenable \ca{s} \cite{Phi00ClassNuclPISimple}.

Our methods are very general and thus ought to provide useful information in many other contexts,
since the existence of $C^*$-cores does not assume that the $L^p$-operator algebra is constructed
from any combinatorial object. 

%==========================================================================================
\subsection*{Acknowledgements}
This paper grew out of discussions held during a visit of the first named
author to the other two at the Mathematical Institute at the University of 
M\"unster in April 2018.
He thanks the members of the Institute for their hospitality during this visit.
The second named author wishes to thank Michal Doucha for very valuable email correspondence on some topics of geometric group theory, and David Kerr for pointing out 
the reference \cite{Ram82TopMsrdGrps} to us.

%==========================================================================================
%==========================================================================================
\section{\texorpdfstring{$C^*$}{C*}-cores in \texorpdfstring{$L^p$}{Lp}-operator algebras}

%==========================================================================================
Let $A$ be a unital Banach algebra.
(We only consider complex Banach algebras.)
Recall that an element $a$ in $A$ is said to be \emph{hermitian} if $\|e^{ita}\|=1$ for all $t\in\mathbb{R}$.
We use $A_{\mathrm{h}}$ to denote the set of hermitian elements in $A$, which is a closed, real linear subspace of $A$ satisfying 
$A_{\mathrm{h}}\cap iA_{\mathrm{h}}=\{0\}$;
see~\cite[Section~5]{BonDun71NumRanges} for details.

If $A$ is a unital \ca{}, then $A_{\mathrm{h}}$ consists precisely of the self-adjoint elements in $A$.
It follows that $A=A_{\mathrm{h}}+iA_{\mathrm{h}}$.
The Vidav--Palmer theorem, \cite[Theorem~6.9]{BonDun71NumRanges}, shows that the converse also holds. More precisely, if $A$ is a unital Banach algebra satisfying $A=A_{\mathrm{h}}+iA_{\mathrm{h}}$, then 
the real-linear involution given by $x+iy\mapsto x-iy$ for $x,y\in A_{\mathrm{h}}$ 
is both isometric and an algebra involution which satisfies the 
$C^*$-identity (namely $\|a^*a\|=\|a\|^2$ for all $a\in A$).
These observations justify the following terminology.

%==========================================================================================
\begin{dfn}
Let $A$ be a unital Banach algebra, and let $B\subseteq A$ be a unital, closed subalgebra.
We say that $B$ is a \emph{unital \casub{} of $A$} if $B=B_{\mathrm{h}}+iB_{\mathrm{h}}$.
\end{dfn}

%==========================================================================================
The following result is standard, and will be needed later.

%==========================================================================================
\begin{lma}
\label{prp:HermSubalg}
Let $A$ be a unital Banach algebra, and let $B\subseteq A$ be a unital, closed subalgebra.
Then $B_{\mathrm{h}}=B\cap A_{\mathrm{h}}$.
In particular, if $A_{\mathrm{h}}$ is closed under multiplication, then so is $B_{\mathrm{h}}$.
\end{lma}

%==========================================================================================
Let $A$ be a unital Banach algebra.
In general, $A_{\mathrm{h}}+iA_{\mathrm{h}}$ is not a subalgebra, since
it is not necessarily closed under multiplication. However, if this is the case, then it follows from \autoref{prp:HermSubalg} that it is the largest unital \casub{} of $A$.
When $A_{\mathrm{h}}$ is itself closed under multiplication, we can say
even more.

%==========================================================================================
\begin{prp}
\label{prp:largestSubCa}
Let $A$ be a unital Banach algebra.
Assume that $A_{\mathrm{h}}$ is closed under multiplication. %for all $a\in A_{\mathrm{h}}$.
Then $A_{\mathrm{h}}+iA_{\mathrm{h}}$ is a commutative, unital \casub{} of $A$.
Moreover, if $C\subseteq A$ is a unital \casub{}, then $C\subseteq A_{\mathrm{h}}+iA_{\mathrm{h}}$.
\end{prp}
\begin{proof}
Since $A_{\mathrm{h}}$ is closed under multiplication, elementary
algebra shows that the subspace $D=A_{\mathrm{h}}+iA_{\mathrm{h}}$ is
also closed under multiplication, and is thus a subalgebra of $A$.
Hence $D$ is the largest unital \casub{} of $B$.

We now show that $D$ is commutative. Given $a,b\in A_{\mathrm{h}}$, the 
element $i(ab-ba)$ is also hermitian by Lemma~5.4 in~\cite{BonDun71NumRanges}. Therefore, since $A_{\mathrm{h}}$ is a $\mathbb{R}$-linear
subspace and is closed under multiplication, $ab-ba$ belongs
to $A_{\mathrm{h}}\cap iA_{\mathrm{h}}=\{0\}$.
Thus $ab=ba$ for all $a,b\in A_{\mathrm{h}}$, and the result follows.
\end{proof}

%==========================================================================================
It is well-known that hermitian elements are preserved by unital, contractive homomorphisms.
The next result, which is probably well-known but which we could not find in the literature, shows that multiplicativity of the map is not needed.
It shows in particular that conditional expectations onto unital subalgebras preserve hermitian elements;
see \autoref{prp:coreCondExp}.

%==========================================================================================
\begin{lma}
\label{prp:HermUnitalMaps}
Let $A$ be a unital Banach algebra, let $B$ be a unital Banach algebra, and let $\varphi\colon A\to B$ be a unital, contractive linear map.
Then $\varphi(A_{\mathrm{h}})\subseteq B_{\mathrm{h}}$.
\end{lma}
\begin{proof}
Given a unital Banach algebra $C$ and $x\in C$, recall that the numerical range of $x$ with respect to $C$ is defined as
\[
V(C,x) = \big\{ f(x)\colon f\in C^*, f(1_C)=1=\|f\| \big\}.
\]
We will use the standard fact that $x\in C$ is hermitian if and only if $V(C,x)\subseteq\RR$.

Let $a\in A_{\mathrm{h}}$.
Let $f\in B^*$ satisfy $f(1_B)=1=\|f\|$, and set $\bar{f}=f\circ\varphi$.
Since $\varphi$ is unital, we have $\bar{f}(1_A)=1$.
Since $\varphi$ is contractive, we have $\|\bar{f}\|\leq 1$ and thus $\|\bar{f}\|=1$.
Then
\[
f(\varphi(a))=\bar{f}(a)\in V(A,a)\subseteq\RR,
\]
and consequently $V(B,\varphi(a))\subseteq\RR$, which implies that $\varphi(a)$ is hermitian.
\end{proof}

%==========================================================================================
Our next step is to describe all hermitian operators on an $L^p$-space for $p\neq 2$; 
see~\autoref{eg:HermBLp}.
Although it would suffice for many concrete examples to only consider $\ell^p$ and $L^p[0,1]$, the proofs are no harder for general $L^p$-spaces, and the extra generality may be useful for future work, as ultraproduct arguments often lead to representations on ``large'' $L^p$-spaces.

To formulate the precise result, we first recall some notions from measure theory.
Recall that a \emph{measure algebra} $(\frakA,\mu)$ is a $\sigma$-complete Boolean algebra $\frakA$ 
together with a $\sigma$-additive map $\mu\colon\frakA\to[0,\infty]$ that satisfies $\mu^{-1}(0)=\{0\}$;
see \cite[Definition~321A, p.68]{Fre-MsrThy3a}.
Given a measure space $(X,\Sigma,\mu)$, the family of null-sets $\mathcal{N}=\{E\in\Sigma\colon \mu(E)=0\}$ 
is a $\sigma$-ideal in $\Sigma$, and $\frakA=\Sigma/\mathcal{N}$ is a $\sigma$-complete Boolean algebra.
Further, the measure $\mu$ induces a map $\bar{\mu}\colon\frakA\to[0,\infty]$ given by $\bar{\mu}(E+\mathcal{N})=\mu(E)$ 
for all $E\in \Sigma$. Moreover, $(\frakA,\bar{\mu})$ is a measure algebra, called the measure algebra associated to $(X,\Sigma,\mu)$;
see \cite[321H, p.69f]{Fre-MsrThy3a}.
%Moreover, every measure algebra arises this way. % \cite[321J, p.70]{Fre-MsrThy3a}.

There are natural notions of measurable and integrable functions on a measure algebra $(\frakA,\mu)$, and one obtains $L^p$-spaces $L^p(\frakA,\mu)$ for every $p\in[1,\infty]$.
If $(\frakA,\bar{\mu})$ is the measure algebra associated to a measure space $(X,\Sigma,\mu)$, then there are natural isometric isomorphisms $L^p(\frakA,\bar{\mu})\cong L^p(X,\Sigma,\mu)$ for each $p\in[1,\infty]$;
see Corollary~363I %p.85
and Theorem~366B %p.131
in \cite{Fre-MsrThy3b}.

A measure algebra $(\frakA,\mu)$ is said to be \emph{semi-finite} if for every $E\in \frakA$ with $\mu(E)=\infty$ there exists a nonzero $E'\leq E$ with $\mu(E')<\infty$.
It is said to be \emph{localizable} if it is semi-finite and $\frakA$ is a complete lattice;
see \cite[Definitions~322A]{Fre-MsrThy3a}.

%==========================================================================================
\begin{rmk}
Localizable measure algebras form the largest class of measure algebras where the Radon-Nikodym theorem is applicable.
Importantly for us, Lamperti's description of the
invertible isometries of an $L^p$-space for $p\neq 2$ from \cite{Lam58IsoLp}, which was originally 
proved only for $\sigma$-finite spaces, remains valid in the more general context of localizable measure algebras;
see Section~3 in~\cite{GarThi22IsoConv}.
\end{rmk}

%==========================================================================================
Given a measure algebra $(\frakA,\mu)$, there is a canonical way to associate to it a semi-finite measure algebra, which can then be Dedekind-completed to obtain a localizable measure algebra.
Both operations identify the associated $L^p$-spaces for $p\in[1,\infty)$ (but not necessarily for $p=\infty$);
see \cite[322P, 322X(a), p.78f]{Fre-MsrThy3a} and \cite[365X(o), 366X(e), p.129, p.139]{Fre-MsrThy3b}.
In particular, we deduce the following:

%==========================================================================================
\begin{prp}
\label{prp:RealizeLpByLocalizable}
Let $(X,\Sigma,\mu)$ be a measure space.
Then there is a (naturally associated) localizable measure algebra $(\frakA,\bar{\mu})$ such that $L^p(X,\Sigma,\mu)$ is isometrically isomorphic to $L^p(\frakA,\bar{\mu})$ for every $p\in[1,\infty)$.
\end{prp}

%==========================================================================================
By \cite[Theorem~322B, p.72]{Fre-MsrThy3a}, the measure algebra associated to a measure space is localizable if and only if the measure space is localizable (in the sense of \cite[Definition~211G, p.13]{Fre-MsrThy2}).
Since every measure algebra is realized by some measure space, we also deduce that for every measure space $\mu$ there exists a localizable measure space $\bar{\mu}$ such that $L^p(\mu)\cong L^p(\bar{\mu})$ for every $p\in[1,\infty)$.

The following result is probably known, but we could only locate it in the literature for the case that the measure space is atomic (\cite[Theorem~2]{Tam69IsomFctnSp}), or $\sigma$-finite (see for example \cite[Lemma~5.2]{GarThi20ExtendingRepr}).
We include here a proof in the general case for the convenience of the reader. 

%==========================================================================================
\begin{prp}
\label{eg:HermBLp}
Let $p\in [1,\infty)\setminus\{2\}$, let $(\frakA,\mu)$ be a localizable measure algebra, and let $a\in\Bdd(L^p(\mu))$.
Then $a$ is hermitian if and only if there exists $h\in L^\infty_{\RR}(\mu)$ such that $a$ is the multiplication operator 
associated to $h$. 
\end{prp}
\begin{proof}
Given $f\in L^\infty(\mu)$, let $m_f\in\Bdd(L^p(\mu))$ denote the associated multiplication operator.
The resulting map $m\colon L^\infty(\mu)\to\Bdd(L^p(\mu))$ is unital and contractive
(and in fact isometric, since $(\frakA,\mu)$ is localizable),
which implies that it preserves hermitian elements;
see \autoref{prp:HermUnitalMaps}.
We have $L^\infty(\mu)_{\mathrm{h}}= L^\infty_{\RR}(\mu)$, and thus every function $f\in L^\infty_{\RR}(\mu)$ defines a hermitian multiplication operator $m_f$.

Conversely, assume that $a$ is hermitian.
We may assume that $a\neq 0$; by rescaling if necessary, we may also assume that $\|a\|\leq\tfrac{\pi}{2}$.
Then, for $t\in\mathbb{R}$, set $u_t=e^{ita}\in \Bdd(L^p(\mu))$.
Then $\|u_t\|\leq 1$, since $a$ is hermitian.
Moreover, $u_tu_{-t}=u_{-t}u_t=\id_{L^p(\mu)}$, which implies that $u_t\colon L^p(\mu)\to L^p(\mu)$ is a surjective isometry.
Moreover, the resulting map $[0,1]\to\mathrm{Isom}(L^p(\mu))$ into the group of surjective isometries, given by $t\mapsto u_t$, 
is norm-continuous.

By Lamperti's theorem (in the form given in Theorem~3.7 in~\cite{GarThi22IsoConv};
see \cite{Lam58IsoLp} for the original statement), for every $t\in [0,1]$ there exist a unique $h_t$ in the unitary group $\mathcal{U}(L^\infty(\mu))$ of $L^\infty(\mu)$, 
and a unique Boolean automorphism $\Phi_t$ of $\frakA$ such that, in the notation of Lemma~3.3 of~\cite{GarThi22IsoConv}, we have $u_t=m_{h_t}\circ v_{\Phi_t}$.
By the norm computation in equation~(6) of~\cite{GarThi22IsoConv}, for $s,t\in [0,1]$ we have 
\[
\|u_t-u_s\| = \max \big\{ \|h_t-h_s\|_\infty, 2(1-\delta_{\Phi_t,\Phi_s}) \big\}.
\]
Since $t\mapsto u_t$ is norm-continuous, it follows that $\Phi_t=\Phi_s$ for all $t,s\in [0,1]$.
Since $\Phi_0$ is the identity automorphism, we deduce that $\Phi_t=\id_{\frakA}$ for all $t\in [0,1]$.
Hence, $\exp(ia)=u_1=m_{h_1}$.
Set 
\[
T= \big\{ it\colon t\in  [-\tfrac{\pi}{2},\tfrac{\pi}{2}] \big\} \ \mbox{ and } \ 
P= \big\{ z\in S^1\colon \mathrm{Re}(z)\geq 0 \big\},
\]
and note that the exponential map induces a bijection from $T$ to $P$.
We let $\log\colon P\to T$ denote the inverse of this map, which is analytic on a neighborhood of $P$.

Since $\|a\|\leq\tfrac{\pi}{2}$, the spectrum of $ia$ is contained in $T$.
Consequently, the spectrum of $u_1$ is contained in $P$. 
Applying analytic functional calculus to $u_1$ we get $ia=\log(u_1)$.
Since $m\colon L^\infty(\mu)\to\Bdd(L^p(\mu))$ is a unital homomorphism, we obtain
\[
a 
= -i\log(u_1)
= -i\log(m_{h_1})
=  m_{-i\log(h_1)}.
\]
Note that $-i\log(h_1)$ belongs to $L^\infty_{\RR}(\mu)$, which finishes the proof.
\end{proof}

%==========================================================================================
\begin{cor}
\label{prp:HermBLpSubalg}
Let $p\in [1,\infty)\setminus\{2\}$, and let $(X,\Sigma,\mu)$ be any measure space.
Then $\Bdd(L^p(\mu))_{\mathrm{h}}$ is closed under multiplication.
\end{cor}
\begin{proof}
Apply \autoref{prp:RealizeLpByLocalizable} to obtain a localizable measure algebra $(\frakA,\bar{\mu})$ such that $L^p(\mu)\cong L^p(\bar{\mu})$.
Then $\Bdd(L^p(\mu))$ and $\Bdd(L^p(\bar{\mu}))$ are isometrically isomorphic as Banach algebras, and the result follows from \autoref{eg:HermBLp}.
\end{proof}

%==========================================================================================
We have arrived at one of the main results of this section:
every unital 
$L^p$-operator algebra contains a largest \casub. 

%==========================================================================================
\begin{thm}
\label{thm:C*core}
Let $p\in [1,\infty)$, and let $A$ be a unital $L^p$-operator algebra.
Set $\core(A):=A_{\mathrm{h}}+iA_{\mathrm{h}}$.
Then $\core(A)$ is the largest unital \casub{} of $A$. 
If $p\neq 2$, then $\core(A)$ is commutative.
\end{thm}
\begin{proof}
Let $\varphi\colon A\to \Bdd(L^p(\mu))$ be an isometric representation of $A$ on some $L^p$-space $L^p(\mu)$.
Since $\varphi(1)$ is a contractive idempotent on $L^p(\mu)$, its 
image is isometrically isomorphic to an $L^p$-space by Theorem~6 in~\cite{Tza69CntrProjLp}.
Thus, upon replacing $L^p(\mu)$ with the image of $\varphi(1)$, we may assume that $A$ is a unital, closed subalgebra of $\Bdd(L^p(\mu))$.

For $p\neq 2$, the result follows by combining \autoref{prp:HermSubalg}, \autoref{prp:largestSubCa} and \autoref{prp:HermBLpSubalg}.
On the other hand, the result is standard for $p=2$, and we include the short argument:
If $A\subseteq \Bdd(L^2(\mu)))$ is unital, then
$\core(A)$ is a subset of the intersection 
$A\cap A^* \subseteq \Bdd(L^2(\mu))$.
On the other hand, $A\cap A^*$ is a unital \casub{} of $A$, and hence
\[
A\cap A^* = (A\cap A^*)_{\mathrm{h}}+i(A\cap A^*)_{\mathrm{h}} \subseteq A_{\mathrm{h}}+iA_{\mathrm{h}} = \core(A).
\]
Thus, $\core(A)=A\cap A^*$, which is therefore the largest unital \casub{} of $A$.
\end{proof}

%==========================================================================================
\begin{dfn}
Let $p\in [1,\infty)$, and let $A$ be a unital $L^p$-operator algebra.
We call the algebra $\core(A):=A_{\mathrm{h}}+iA_{\mathrm{h}}$ the \emph{$C^*$-core} of $A$.
\end{dfn}

%==========================================================================================
\begin{exa}
\label{exa:BLp}
Let $(X,\mu)$ be a localizable measure space and let $p\in [1,\infty)\setminus\{2\}$.
Then $\core(\Bdd(L^p(\mu)))=\{m_f\colon f\in L^\infty(\mu)\}\cong L^\infty(\mu)$, the algebra of multiplication operators.
\end{exa}

%==========================================================================================
\begin{rmk}
Let $(X,\mu)$ be a localizable measure space, let $p\in [1,\infty)\setminus\{2\}$, and let $A\subseteq \Bdd(L^p(\mu))$ be a closed subalgebra.
Then $A_{\mathrm{h}}=A\cap L^\infty_\RR(\mu)$, and thus
\[\mathrm{core}(A)=\big(A\cap L^\infty_\RR(\mu)\big)+ i\big(A\cap L^\infty_\RR(\mu)\big),\]
which can be strictly smaller than $A\cap L^\infty(\mu)$. This is the case, for example, 
for the disc algebra 
\[A(\mathbb{D})=\{f\in C(\mathbb{D})\colon f|_{\mathbb{D}^\circ} \mbox{ is holomorphic}\}.\]
Indeed, 
since $A(\mathbb{D})$ is a Banach subalgebra of 
$L^\infty(\mathbb{D})$, it is in particular an $L^p$-operator algebra for every $p\in [1,\infty)$. On
the other hand, $A(\mathbb{D})\cap L^\infty_{\mathbb{R}}(\mathbb{D})=\{0\}$, and 
thus $\core(A(\mathbb{D}))=\{0\}$, although 
$A(\mathbb{D})\cap L^\infty(\mathbb{D})=A(\mathbb{D})$.\end{rmk}

%==========================================================================================
The next result follows directly from \autoref{prp:HermUnitalMaps}.

%==========================================================================================
\begin{prp}
\label{prp:CorePreserved}
Let $p,q\in [1,\infty)$, let $A$ be a unital $L^p$-operator algebra, let $B$ be a unital $L^q$-operator algebra, and let
$\varphi\colon A\to B$ be a unital, contractive, linear map. 
Then $\varphi(\core(A))\subseteq\core(B)$, and 
$\varphi\colon \core(A)\to \core(B)$ is a $\ast$-homomorphism. 
\end{prp}

%==========================================================================================
\begin{rmk}
\autoref{prp:CorePreserved} does not generalize to non-unital maps, even if they are multiplicative: for 
$p\neq 2$,
consider the homomorphism $\CC\to M_2=\Bdd(\ell^p(\{0,1\})$ determined by sending the unit to a contractive, non-hermitian idempotent, such as $e=\frac{1}{2} \left( \begin{smallmatrix}
    1 & 1 \\
    1 & 1
\end{smallmatrix}\right)$. Indeed, since the 
idempotent $e$ has the form $e=\frac{1}{2}(I+U)$, where $I$ is the unit in $M_2$ and $U$ is the invertible isometry $U=\left( \begin{smallmatrix}
    0 & 1 \\
    1 & 0
\end{smallmatrix}\right)$, it follows from the 
triangle inequality that $e$ is contractive. 
Moreover, $e$ is not hermitian for $p\neq 2$, 
since it does 
not belong to $\ell^\infty(\{0,1\})$.
\end{rmk}

%==========================================================================================
\begin{dfn}
\label{df:CondExp}
Given a unital Banach algebra $A$ and a unital, closed subalgebra $B\subseteq A$, a \emph{conditional expectation} from $A$ onto $B$ is a unital, contractive, linear map $E\colon A\to B$ satisfying $E(b_1ab_2)=b_1E(a)b_2$ for all $a\in A$ and $b_1,b_2\in B$. (In particular, $E(b)=b$ for all $b\in B$.)
\end{dfn}

%==========================================================================================
%\begin{rmk}
%In the above definition, note that $E(b)=b$ for all $b\in B$. 
%\end{rmk}

%==========================================================================================
The notion of a conditional expectation is well-established for \ca{s}, and generalizations to Banach algebras such as the one above (but also variations thereof) have been considered in several places;
see for example \cite{LauLoy08ContrProjBAlg}.%,LauLoy10CorrigContrProjBAlg}.

We record the following fact for future use.
It is an immediate consequence of \autoref{prp:CorePreserved}, since conditional expectations are unital and contractive.

%==========================================================================================
\begin{prp}
\label{prp:coreCondExp}
Let $A$ be a unital $L^p$-operator algebra, let $B\subseteq A$ be a unital, closed subalgebra, and let $E\colon A\to B$ be a conditional expectation.
Then $E(\core(A))=\core(B)$, and thus $E$ restricts to a conditional expection between the respective $C^*$-cores.
\end{prp}

%==========================================================================================
We end this section by exploring $C^*$-cores in reduced crossed products. 
First, we recall some elementary facts from \cite{Phi13arX:LpCrProd}, whose notation we follow.

%==========================================================================================
\begin{pgr}
\label{pgr:crProduct}
Let $G$ be a discrete group, let $A$ be a unital Banach algebra, and let $\alpha\colon G\to\Aut(A)$ be an action by isometric isomorphisms.
We use $C_c(G,A,\alpha)$ to denote the complex algebra of functions $G\to A$ with finite support.
Given $a\in A$ and $g\in G$, we let $a u_g\in C_c(G,A,\alpha)$ be the function that maps $g$ to $a$ and everything else to $0$.
We write $u_g$ for $1 u_g$, and observe that any element in $C_c(G,A,\alpha)$ can be written uniquely as $\sum_{g\in G} a_g u_g$, where all but finitely many $a_g\in A$ are zero.

The product in $C_c(G, A,\alpha)$ is determined by the (formal) rules $u_g u_h=u_{gh}$ and $u_g a u_{g^{-1}} = \alpha_g(a)$, for $g,h\in G$ and $a\in A$.
In particular, $u_1$ is the unit of $C_c(G,A,\alpha)$. 
Moreover, we have canonical unital homomorphisms $C_c(G)\to C_c(G,A,\alpha)$ and $A\to C_c(G,A,\alpha)$ given by $u_g\mapsto u_g$ and $a\mapsto au_1$.

A \emph{representation} of $(G,A,\alpha)$ on an $L^p$-space $E$ is a pair $(\pi,v)$ where $\pi\colon A\to\Bdd(E)$ is a unital, contractive homomorphism and $v\colon G\to\Isom(E)$ is an isometric representation of $G$, satisfying
$v_g\pi(a)v_{g^{-1}} = \pi(\alpha_g(a))$ for all $g\in G$ and $a\in A$.
We write $\mathrm{Rep}_p(G,A,\alpha)$ for the class of representations
of $(G,A,\alpha)$ on $L^p$-spaces.
Given $(\pi,v)\in \Rep_p(G,A,\alpha)$ as above, there is 
a unital homomorphism $\pi\rtimes v\colon C_c(G,A,\alpha)\to\Bdd(E)$ given by
$(\pi\rtimes v)(au_g) = \pi(a)v_g$
for $a\in A$ and $g\in G$.

Next, we recall the construction of regular representations of $(G,A,\alpha)$.
Let $\pi_0\colon A\to B(L^p(\mu))$ be a unital, contractive representation
on an $L^p$-space $L^p(\mu)$.
Let $c_G$ denote the counting measure on $G$.
As in \cite[Lemma~2.10]{Phi13arX:LpCrProd}, we identify 
$L^p(c_G\times\mu)$ with $\ell^p(G,L^p(\mu))$. 
By \cite[Lemma~2.11]{Phi13arX:LpCrProd}, the representation $\lambda^\mu$ of $G$ on $\ell^p(G,L^p(\mu))$ given by
$\lambda^\mu=\lambda\otimes \id_{L^p(\mu)}$ is isometric.
Further, there is a unital, contractive representation $\pi$ of $A$ on $\ell^p(G,L^p(\mu))$ given by
\[
(\pi(a)\xi)(g)
= \pi_0(\alpha_g^{-1}(a))(\xi(g)),
\]
for $g\in G$, $a\in A$ and $\xi\in \ell^p(G,L^p(\mu))$.
Then $(\pi,\lambda^\mu)$ is a covariant representation, and $\pi\rtimes \lambda^\mu$ is called the \emph{regular representation} induced by~$\pi_0$. We write $\mathrm{RegRep}_p(G,A,\alpha)$ for the class of 
regular representations of $(G,A,\alpha)$ on $L^p$-spaces.
\end{pgr}

%==========================================================================================
%Using regular representations as above, one defines an $L^p$-operator norm on $C_c(G,A,\alpha)$ and then completes this object to obtain the reduced crossed product. 
%The next definition follows Definition~3.3 in \cite{Phi13arX:LpCrProd}, but we do not restrict to $\sigma$-finite measure spaces.

%==========================================================================================
\begin{dfn}\label{df:CrossedProds}
Given $f\in C_c(G,A,\alpha)$, set 
\begin{align*}
 \|f\|&= \sup\big\{\|(\pi\rtimes v)(f)\|\colon (\pi,v)\in\Rep_p(G,A,\alpha)\big\},\ \ \mbox{and} \\
 \|f\|_\lambda&= \sup\big\{\|(\pi\rtimes \lambda^\mu)(f)\|\colon (\pi,\lambda^\mu)\in\mathrm{RegRep}_p(G,A,\alpha)\big\}.
\end{align*}
The \emph{full $L^p$-operator crossed product} of $(G,A,\alpha)$, denoted by $F^p(G,A,\alpha)$, is the completion of $C_c(G,A,\alpha)$ in the norm $\|\cdot\|$,
while the \emph{reduced $L^p$-operator crossed product} of $(G,A,\alpha)$, denoted by $F^p_\lambda(G,A,\alpha)$, is the completion of $C_c(G,A,\alpha)$ in the norm
$\|\cdot\|_\lambda$.\end{dfn}

%==========================================================================================
By \cite[Remark~4.6]{Phi13arX:LpCrProd}, the identity on $A$ induces a unital, isometric homomorphism $A\to F^p_\lambda(G,A,\alpha)$, 
which we use to identify $A$ with a unital subalgebra of $F^p_\lambda(G,A,\alpha)$.
We let $E\colon F^p_\lambda(G,A,\alpha)\to A$ be the standard conditional expectation as in \cite[Definition~4.11]{Phi13arX:LpCrProd}, which
satisfies $E(au_g)=a$ if $g=1$, and zero
else.

%==========================================================================================
\begin{thm}
\label{prp:coreCrossedProd}
Let $p\in[1,\infty)\setminus\{2\}$, let $G$ be a discrete group, let $A$ be a unital $L^p$-operator algebra, and let $\alpha\colon G\to\Aut(A)$ be an action.
Then the canonical embedding $A\subseteq F^p_\lambda(G,A,\alpha)$ identifies the $C^*$-core of $F^p_\lambda(G,A,\alpha)$ with that of $A$, that is,
$\core(F^p_\lambda(G,A,\alpha)) = \core(A)$.
\end{thm}
\begin{proof}
Use \cite[Lemma~3.19]{Phi13arX:LpCrProd} to find an isometric, unital representation $\pi_0\colon A\to \Bdd(L^p(\mu))$ such that 
$\pi\rtimes \lambda^\mu$ 
induces the norm of $F^p_\lambda(G,A,\alpha)$. 
By \autoref{prp:RealizeLpByLocalizable} and without loss of generality, we assume that 
$\mu$ is localizable. 
Given $g\in G$ and $\xi\in L^p(\mu)$, we let $\delta_g\otimes\xi$ denote the element in $\ell^p(G,L^p(\mu))$ that maps $g$ to $\xi$ and every other element in $G$ to zero.
Then
\[
\lambda^\mu_g (\delta_h\otimes\xi) = \delta_{gh}\otimes\xi, \quad\text{ and }\quad
\pi(a) (\delta_h\otimes\xi) = \delta_h\otimes \pi_0(\alpha_h^{-1}(a))(\xi),
\]
for $g,h\in G$, $a\in A$ and $\xi\in L^p(\mu)$.
Define $F\colon \Bdd(\ell^p(G,L^p(\mu))) \to \Bdd(L^p(\mu))$ by
\[
F(x)(\xi)= (x(\delta_1\otimes\xi))(1),
\]
for $x\in \Bdd(\ell^p(G,L^p(\mu))$ and $\xi\in L^p(\mu)$.

\textbf{Claim~1:}\emph{ 
We have $F\circ(\pi\rtimes \lambda^\mu)=\pi_0\circ E$, that is, the following diagram commutes:}
\[
\xymatrix{
F^p_\lambda(G,A,\alpha) \ar[rr]^-{\pi\rtimes \lambda^\mu} \ar[d]_{E}
&&  \Bdd(\ell^p(G,L^p(\mu))) \ar[d]^{F} \\
A \ar[rr]_-{\pi_0}
&& \Bdd(L^p(\mu)).
}
\]
Since all the maps involved are continuous, it is enough to verify the equality on $C_c(G,A,\alpha)$.
Let $a=\sum_{g\in G} a_g u_g \in C_c(G,A,\alpha)$, and let $\xi\in L^p(\mu)$.
Then
\begin{align*}
(\pi_0\circ E)(a)\xi = \pi_0(a_1)\xi.
\end{align*}
On the other hand, we have
\begin{align*}
(F\circ(\pi\rtimes \lambda^\mu))(a)\xi
&= (\pi\rtimes \lambda^\mu)(a)(\delta_1\otimes\xi)(1)
= \Big( \sum_{g\in G} \pi(a_g) \lambda^\mu_g \Big)(\delta_1\otimes\xi)(1) \\
&= \Big( \sum_{g\in G} \delta_g\otimes  \pi_0(\alpha_g^{-1}(a_g))(\xi) \Big)(1)
= \pi_0(a_1)\xi,
\end{align*}
as desired.

For the next two claims, we fix $g\in G\setminus\{1\}$.

\textbf{Claim~2:}
\emph{If $\eta\colon G\to L^\infty(\mu)$ is a bounded function with associated multiplication operator $m_\eta\in\Bdd(\ell^p(G,L^p(\mu))$,
then $F(m_\eta \lambda^\mu_g )=0$.}
To prove the claim, let
$\xi\in L^p(\mu)$. Then 
\begin{align*}
F(m_\eta \lambda^\mu_g)(\xi)&= (m_\eta \lambda^\mu_g (\delta_1\otimes\xi))(1)
= (m_\eta(\delta_g\otimes\xi))(1)\\
&= (\delta_g\otimes ( m_{\eta(g)} \xi))(1)=0.
\end{align*}

\textbf{Claim~3:}
\emph{Let $a\in\core( F^p_\lambda(G,A,\alpha) )$.
Then $E(au_g )=0$.}
Note that $(\pi\rtimes v)(a)$ belongs to the $C^*$-core of $\Bdd(\ell^p(G,L^p(\mu))$, since $\pi\rtimes v$
is a contractive, unital map.
By \autoref{exa:BLp}, and since $\mu$ is localizable, there exists a bounded function $\eta\colon G\to L^\infty(\mu)$ such that $m_\eta=(\pi\rtimes v)(a)$.
Using Claim~1 at the first step, and using Claim~2 at the last step, we get
%Using this at the last step (claim~2), and using claim~1 at the first step, we obtain
\begin{align*}
\pi_0(E(au_g))
&= F( (\pi\rtimes v)(au_g) )
= F( (\pi\rtimes v)(a) \lambda^\mu_g )= F(m_\eta \lambda^\mu_g)= 0.
\end{align*}
Since $\pi_0$ is isometric, the claim follows.

Let $a\in\core( F^p_\lambda(G,A,\alpha) )$.
We want to show that $a=E(a)$.
By \autoref{prp:coreCondExp}, we have $E(a)\in\core(A)\subseteq\core( F^p_\lambda(G,A,\alpha) )$.
Thus, for each $g\in G\setminus\{1\}$, we have
\[
E(au_g) = 0 = E(E(a)u_g).
\]
For $g=1$, we have $E(au_1)=E(a)=E(E(a)u_1)$.
Since $E$ is faithful (see \cite[Proposition~4.9]{Phi13arX:LpCrProd}), it follows that $a=E(a)$, as desired.	
\end{proof}

%==========================================================================================
\begin{cor}
Let $G\curvearrowright X$ be a topological action of a discrete group $G$ on a compact, Hausdorff space $X$.
Then $\core(F^p_\lambda(G,X))=C(X)$. 
\end{cor}

%==========================================================================================
Given a discrete group $G$, we use $F^p_\lambda(G)$ to denote its reduced group $L^p$-operator algebra (\cite{Phi13arX:LpCrProd}), which was originally introduced by Herz as the `algebra of $p$-pseudofunctions' (see also \cite[Definition~3.1]{GarThi15GpAlgLp}).
We have $F^p_\lambda(G)\cong F^p_\lambda(G,\{\ast\})$. 

%==========================================================================================
\begin{cor}
\label{prp:coreGpAlg}
Let $G$ be a discrete group.
Then $\core( F^p_\lambda(G) ) = \CC$.
\end{cor}

%==========================================================================================
\begin{pbm}
Given a countable, discrete group $G$, determine the $C^*$-core of its full group $L^p$-operator algebra. Can one give an explicit description for $G=\mathbb{F}_n$?
\end{pbm}

%==========================================================================================
Another fundamental tool for the computation of $C^*$-cores will be given in \autoref{prp:coreGpd}.
We postpone the discussion of further examples until then.

%==========================================================================================
%==========================================================================================
\section{The Weyl groupoid of an \texorpdfstring{$L^p$}{Lp}-operator algebra}

%==========================================================================================
From now on and until the end of this section, we fix $p\in [1,\infty)\setminus\{2\}$.
Thus, given a unital $L^p$-operator algebra $A$, its core is a commutative, unital \ca{} by \autoref{thm:C*core},
and we write $X_A$ for its spectrum, which is a compact Hausdorff space.
Under this identification, we will regard $C(X_A)$ as a unital subalgebra of $A$.

In this section, we first construct a canonical inverse semigroup of partial homeomorphisms on $X_A$;
see \autoref{cor:RealizableHomInvSmgp}.
The associated groupoid of germs, which we denote by $\calG_A$, is a topologically principal, \'etale groupoid with unit space $X_A$,
which we call the \emph{Weyl groupoid} of $A$;
see \autoref{df:WeylGroupoid}.
The Weyl groupoid contains information about the internal dynamics of the algebra $A$.
This can be best seen in the case of crossed products:
we will show that for topologically free actions,
the Weyl groupoid of the crossed product can be naturally identified with the 
transformation groupoid (see \autoref{thm:GroupoidRigidity} and the remarks at the beginning
of Section~6 for the details).

For the next definition, we use $C(X_A)_+$ to denote the set of continuous functions $X_A\to[0,\infty)$.
Note that $C(X_A)_+$ is the set of positive hermitian elements in $A$.

%==========================================================================================
\begin{dfn}
\label{df:Normalizer}
Let $A$ be a unital $L^p$-operator algebra.
Given open subsets $U,V\subseteq X_A$ and a homeomorphism $\alpha\colon U\to V$,
we say that $\alpha$ is \emph{realizable (within $A$)} if there exist $a,b\in A$ satisfying the following conditions.
\begin{enumerate}
\item
For all $f\in C(X_A)_+$, we have $afb, bfa\in C(X_A)_+$.
\item We have
$U=\{x\in X_A\colon ba(x)>0\}$ and $V=\{x\in X_A\colon ab(x)>0\}$.
\item
For all $x\in U$, all $y\in V$, all $f\in C_0(V)$ and all $g\in C_0(U)$, we have
\[
f(\alpha(x))ba(x)=bfa(x) \ \ \mbox{ and } \ \ g(\alpha^{-1}(y))ab(y)=agb(y).
\]
\end{enumerate}
In this case, we say that $s=(a,b)$ is an \emph{admissible pair} which \emph{realizes} $\alpha$, and
we write $\alpha=\alpha_{s}$, $U=U_s$ and $V=V_s$. 
\end{dfn}

%==========================================================================================
Realizable pairs as in the definition above will play the role of the 
normalizers used by Renault in \cite{Ren08Cartan}.
Indeed, a pair $(a,b)$ replaces what in Renault's context would be a pair of the form 
$(a,a^*)$ where $a$ is a normalizer.
In our setting, however, there are a number of difficulties arising from the absence of a canonical involution on a general $L^p$-operator algebra. 

%==========================================================================================
\begin{prp}
\label{lma:NormalizerBasic}
Let $A$ be a unital $L^p$-operator algebra, and let $s=(a,b)$ and $t=(c,d)$ be admissible pairs in $A$.
\begin{enumerate}
\item[(a)] 
The inverse of $\alpha_s$ is realized by the \emph{reverse} of $s$, which is defined to be the admissible pair $s^\sharp=(b,a)$.  
\item[(b)]
The product $st=(ac,db)$ realizes the composition
\[
\alpha_s\circ \alpha_t|_{U_t\cap \alpha_t^{-1}(U_s)}\colon U_t\cap \alpha_t^{-1}(U_s)\to V_s\cap \alpha_s(V_t).
\]
\item[(c)]
For every $f\in C(X_A)$, the pair $s_f=(f,\overline{f})$ is admissible and $\alpha_{s_f}=\id_{U_{s_f}}$.
In particular, the identity map on every open subset of $X_A$ is realizable.
\end{enumerate}
\end{prp}
\begin{proof}
Part (a) is immediate from the definition, so we check~(b). 
Condition~(1) in \autoref{df:Normalizer} is readily verified for the pair $(ac,db)$.
Set
\[U_{st}=U_t\cap \alpha_t^{-1}(U_s) \ \ \mbox{ and } \ \ V_{st}=V_s\cap \alpha_s(V_t),\]
which are open subsets of $X_A$. 
We claim that 
\[
U_{st} = \big\{ x\in X_A : dbac(x)>0 \big\} \ \ \mbox{ and } \ \ 
V_{st} = \big\{ x\in X_A : dbac(x)>0 \big\},
\]
which is Condition~(2) in \autoref{df:Normalizer}.
We prove the first equality, since the other one is obtained by considering the reverses of $s$ and $t$.
Set $f=ba$, which is a strictly positive function on $U_s=\{x\in X_A\colon ba(x)>0\}$.
Using condition~(3) of \autoref{df:Normalizer} for the pair $(c,d)$, we get
\[
dbac(x)=dfc(x)=f(\alpha_t(x))dc(x)=ba(\alpha_t(x))dc(x)
\]
for all $x\in X_A$.
Note that the composition $f\circ \alpha_t$ is a strictly positive function in $C_0(U_{st})$.
In particular, the expression above is positive if and only if $ba(\alpha_t(x))>0 $ and $dc(x)>0$, which is equivalent 
to $x\in U_{st}$.
This proves the claim.

It remains to verify Condition~(3) in \autoref{df:Normalizer}; we will only do the first half, since 
the other one is analogous.
Let $x\in U_{st}$ and let $f\in C_0(V_{st})$. In the following computation, we use the identity 
$dbac(x)=ba(\alpha_t(x))dc(x)$ at the first step; the fact that $(a,b)$ realizes $\alpha_s$ at the 
second step; and the fact that $(c,d)$ realizes $\alpha_t$ at the third step (applied to $bfa$
in place of $f$), to get
\begin{align*}
f(\alpha_s(\alpha_t(x)))dbac(x)&=f(\alpha_s(\alpha_t(x)))ba(\alpha_t(x))dc(x)\\
&=bfa(\alpha_t(x))dc(x)
=dbfac(x). 
\end{align*}
This completes the proof.

Finally, part~(c) is immediately checked, using that $C(X_A)$ is commutative.
\end{proof}

%==========================================================================================
For the reader's convenience, and to fix notation and terminology, we include some standard background on inverse semigroups and \'etale groupoids.
Recall that an \emph{inverse semigroup} is a semigroup $S$ together with an involution
$\sharp\colon S\to S$ satisfying $ss^\sharp s=s$ for all $s\in S$.
A inverse subsemigroup of $S$ is a subsemigroup that is closed under the involution.
A typical example of an inverse semigroup is $\Homeo_{\mathrm{par}}(X)$, the set
of partial homeomorphisms of a compact Hausdorff space $X$.

%==========================================================================================
\begin{cor}
\label{cor:RealizableHomInvSmgp}
Let $A$ be a unital $L^p$-operator algebra.
Then the set of realizable partial homeomorphisms on $X_A$ is an inverse subsemigroup of $\Homeo_{\mathrm{par}}(X_A)$.
\end{cor}
\begin{proof}
 This follows immediately from \autoref{lma:NormalizerBasic}.
\end{proof}

%==========================================================================================
We now use the inverse semigroup of realizable partial homeomorphisms on $X_A$ to construct an \'etale groupoid with unit space $X_A$.

%==========================================================================================
\begin{dfn}[see~\cite{Ren80GrpdApproach}]
\label{df:EtaleGroupoid}
A \emph{topological groupoid} is a topological space $\calG$, endowed with a 
distinguished subset $\calG^{(2)}\subseteq\calG\times\calG$ of 
\emph{composable arrows} and continuous operations $\calG^{(2)}\to\calG$ (composition, denoted $(\gamma,\delta)\mapsto \gamma\delta$) and 
$(\cdot)^{-1}\colon\calG\to\calG$ (inversion, denoted $\gamma\mapsto\gamma^{-1}$) satisfying
\begin{enumerate}
\item
If $(\gamma, \eta)$ and $(\eta,\xi)$ belong to $\calG^{(2)}$, then so do
$(\gamma\eta,\xi)$ and $(\gamma,\eta\xi)$ and we have $(\gamma\eta)\cdot \xi=\gamma\cdot (\eta\xi)$.
\item
For all $\gamma\in\calG$ we have $(\gamma^{-1})^{-1}=\gamma$.
\item
For all $\gamma\in\calG$ we have $(\gamma,\gamma^{-1})\in\calG^{(2)}$.
\item
For every $(\gamma,\eta)\in\calG^{(2)}$ we have $\gamma^{-1}(\gamma\eta)=\eta$ and $(\gamma\eta)\eta^{-1}=\gamma$.
\end{enumerate}

The set $\calGZero:= \{ \gamma\in\calG \colon \gamma=\gamma^{-1}=\gamma^2\}$ is called the \emph{unit space} of $\calG$.
The domain and range maps $\dom,\ran\colon\calG\to\calGZero$ are given by
\[
\dom(\gamma) := \gamma^{-1}\gamma, \andSep 
\ran(\gamma) := \gamma\gamma^{-1}
\]
for all $\gamma\in\calG$. 
The groupoid $\calG$ is \emph{\'etale} if the (automatically continuous) domain and
range maps are local homeomorphisms, and \emph{Hausdorff} if $\calG$ is Hausdorff as a
topological space.
\end{dfn}

%==========================================================================================
Let $X$ be a compact Hausdorff space, and let $S\subseteq\Homeo_{\mathrm{par}}(X)$ be an inverse subsemigroup. 
The \emph{groupoid of germs} $\calG(S)$ of $S$ is defined as follows.
On the set
\[
\big\{ (s,x)\in S\times X \colon s\in S, x\in \dom(s) \big\},
\] 
define an equivalence relation by setting $(s,x)\sim (t,y)$ whenever 
$x=y$ and there exists a neighborhood $U$ of $x$ in $X$ such that $s|_U=t|_U$.
We write $[s,x]$ for the equivalence class of $(s,x)$.
Then the quotient $\calG(S)$ by this equivalence relation has a natural groupoid
structure with $\ran([s,x])=s(x)$ and 
$\dom([s,x])=x$, and operations given by
\[
[s,t\cdot y][t,y]=[st,y], \andSep
[s,x]^{-1}=[s^\sharp, s(x)]
\]
for all $s,t\in S$ and all $x,y\in X$.
With the topology defined by the basic open sets
\[
\mathcal{U}_{U,s,V} = \big\{ [s,x] : x\in U, s(x)\in V \big\},
\]
for $U,V\subseteq X$ open and $s\in S$, the groupoid $\calG(S)$ is \'etale.
The unit space of $\calG(S)$ can be canonically identified with $X$, and is therefore compact and Hausdorff.
For details, we refer to \cite[Section~3]{Ren08Cartan}.

The next definition follows Renault's terminology from \cite[Definition~4.11]{Ren08Cartan}:

%==========================================================================================
\begin{dfn}
\label{df:WeylGroupoid}
Let $A$ be a unital $L^p$-operator algebra.
We define the \emph{Weyl groupoid} of $A$, denoted by $\calG_A$, to be the 
groupoid of germs of the inverse subsemigroup of realizable partial homeomorphisms of $X_A$.
\end{dfn}

%==========================================================================================
The Weyl groupoid of an $L^p$-operator algebra is sometimes too small to carry any useful information. 

%==========================================================================================
\begin{exa}
\label{exa:GroupoidGrpAlgs}
Let $G$ be a discrete group, and let $F^p_\lambda(G)$ be its reduced group $L^p$-operator algebra. % and let $p\in [1,\infty)\setminus\{2\}$. 
Then $X_{F^p_\lambda(G)}=\{\ast\}$ by \autoref{prp:coreGpAlg}, 
and thus $\calG_{F^p_\lambda(G)}$ is the trivial groupoid, regardless of $G$.  
\end{exa}

%==========================================================================================
The reason why $\mathcal{G}_A$ remembers so little about 
$A$ in \autoref{exa:GroupoidGrpAlgs} is that the group~$G$, when regarded as a groupoid 
(with $G^{(0)}=\{\ast\}$), has very large stabilizers (or isotropy groups).
The case we will be interested in, namely that of ``small'' stabilizers, is formalized in the following notion.
Given a groupoid $\calG$ and $x\in\calGZero$, the set $x\calG x = \{\gamma\in\calG \colon \ran(\gamma)=\dom(\gamma)=x\}$ is a group, called the \emph{isotropy group} at~$x$.
One says that $x$ has \emph{trivial isotropy} if $x\calG x$ contains only $x$ itself.
The set $\calG' := \{\gamma\in\calG \colon \dom(\gamma)=\ran(\gamma)\}$ is also called the \emph{isotropy bundle}.

%==========================================================================================
\begin{dfn}[{\cite[Definitions~3.4, 3.5]{Ren08Cartan}}]\label{df:topPpal}
An \'etale groupoid $\calG$ is said to be \emph{topologically principal} if the set of points in $\calGZero$ with trivial isotropy is dense in~$\calGZero$;
it is said to be \emph{effective} if the interior of $\calG'$ is $\calGZero$.
\end{dfn}

%==========================================================================================
Let $\calG$ be an \'etale groupoid.
If $\calG$ is topologically principal and Hausdorff, then $\calG$ is effective, and the converse holds under suitable assumptions;
see \cite[Proposition~3.6]{Ren08Cartan}.
The prototypical example of a topologically principal groupoid is the 
transformation groupoid of a topologically free action\footnote{Recall that an action $G\curvearrowright X$ of a discrete group on a topological space $X$ is said to be \emph{topologically free} if for every $g\in G\setminus\{1\}$, the interior of the set $\{x\in X\colon g\cdot x=x\}$ is empty.} of a discrete group.

%==========================================================================================
\begin{rmk}
\label{rem:OpenBisGerms}
(See the beginning of Section~3 of~\cite{Ren08Cartan}.)
Let $\calG$ be an \'etale groupoid, and denote by $\mathcal{S}(\calG)$
the inverse semigroup of its open bisections\footnote{Recall that a subset $S\subseteq \calG$ is said to be an \emph{open bisection} if it is open and the 
restrictions of the source and range maps to $S$ are 
injective.}.
Recall that any $S\in \mathcal{S}(\calG)$ defines a homeomorphism $\beta_S\colon \dom(S)\to \ran(S)$
that satisfies $\beta_S(x)=\ran(Sx)$ for all $x\in \dom(S)$.
Moreover, the induced map $\beta\colon \mathcal{S}(\calG)\to \Homeo_{\mathrm{par}}(\calGZero)$
is an inverse semigroup homomorphism. We let $\mathcal{P}(\calG)$ denote the image of $\beta$. 
By Corollary~3.3 in~\cite{Ren08Cartan}, the groupoid of
germs of $\mathcal{P}(\calG)$ is isomorphic to $\calG$ if and only if $\calG$ is effective.
Moreover, if this is the case, then $\beta$ identifies $\mathcal{S}(\calG)$ bijectively with $\mathcal{P}(\calG)$.
\end{rmk}

%==========================================================================================
Being a groupoid of germs, $\calG_A$ is always effective and \'etale.
For later reference, we record this and other properties of $\calG_A$ in the following proposition.

%==========================================================================================
\begin{prp}
Let $A$ be a unital $L^p$-operator algebra.
Then $\calG_A$ is a locally compact (not necessarily Hausdorff), effective, \'etale groupoid,
and $\calG_A^{(0)}$ is naturally homeomorphic to $X_A$.
\end{prp}
\begin{proof}
It remains to show that $\calG_A$ is locally compact.
This follows using that the range map is a local homeomorphism onto the compact, Hausdorff space $\mathcal{G}_A^{(0)}= X_A$.
\end{proof}

%==========================================================================================
In contrast to \autoref{exa:GroupoidGrpAlgs}, we will show later that when 
$A$ is the reduced $L^p$-groupoid algebra of a topologically principal, Hausdorff groupoid
(in the sense of \cite{GarLup17ReprGrpdLp}; see \autoref{df:GroupoidLpOpAlg} below), 
then $\calG_A$ is a complete invariant for $A$.

%==========================================================================================
%==========================================================================================
\section{Reduced groupoid \texorpdfstring{$L^p$}{Lp}-operator algebras}

%==========================================================================================
Throughout this section, $\calG$ denotes a locally compact, Hausdorff, \'etale groupoid, and $\calGZero$ denotes its unit space.
In this section we recall the construction of the reduced $L^p$-operator algebra of $\calG$ from \cite{GarLup17ReprGrpdLp}, and we prove the basic properties that will be used later on.
Given the absence of $C^*$-algebraic tools such as polar decomposition or continuous functional calculus, we spend some time on technical details.
Our approach here is different from that in \cite{GarLup17ReprGrpdLp}, and
is inspired by the notes of Sims \cite{Sim17arX:EtaleGpds}; see in particular Section~3.3 there. 

We adopt the following notational conventions: for $x\in \calGZero$, we set
\[\calG x=\{\gamma\in\calG\colon \dom(\gamma)=x\} \ \ \mbox{ and } \ \ x\calG=\{\gamma\in\calG\colon \ran(\gamma)=x\}.\]
(Often these sets are denoted in the literature by $\calG_x$ and $\calG^x$, respectively.)

The first step is to define the appropriate convolution algebra.

%==========================================================================================
\begin{dfn}
\label{dfn:convolution}
%Let $\calG$ be a locally compact, Hausdorff, \'etale groupoid.
We denote by $C_c(\calG)$ the space of compactly supported, continuous functions $\calG\to\CC$. 
For $f,g\in C_c(\calG)$, their convolution $f\ast g\colon\calG\to\CC$ is defined by 
\[
(f\ast g)(\gamma)
%= \sum_{\eta\sigma=\gamma}f(\eta)g(\sigma)
= \sum_{\sigma\in\calG\dom(\gamma)}f(\gamma\sigma^{-1})g(\sigma)
= \sum_{\tau\in\ran(\gamma)\calG}f(\tau)g(\tau^{-1}\gamma)
\]
for all $\gamma\in\calG$.
Together with pointwise addition and scalar multiplication, this makes $C_c(\calG)$ into a complex algebra.
Given $f\in C_c(\calG)$ and $\gamma\in\calG$, we define $\delta_\gamma\ast f, f\ast\delta_\gamma\colon\calG\to\CC$ by 
\begin{align*}
(\delta_\gamma\ast f)(\sigma)
&= \begin{cases}
f(\gamma^{-1}\sigma), & \text{if} \ran(\sigma)=\ran(\gamma) \\
0, & \text{otherwise},
\end{cases} \\
(f\ast\delta_\gamma)(\sigma)
&= \begin{cases}
f(\sigma\gamma^{-1}), & \text{if} \dom(\sigma)=\dom(\gamma) \\
0, & \text{otherwise},
\end{cases}
\end{align*}
for $\sigma\in\calG$.
\end{dfn}

%==========================================================================================
Given $x\in\calGZero$, since $\calG$ is \'etale, the relative topology on $\calG x$ is discrete.
Therefore, elements in $C_c(\calG x)$ are finite linear combinations of $\delta_\gamma$ for $\gamma\in\calG x$.
Using this, in the following proposition we show that one can define convolution between elements in $C_c(\calG)$ and $\ell^p(\calG x)$.
Recall that the $I$-norm of $f\in C_c(\calG)$ is given by
\[
\|f\|_I 
= \max \Big\{ 
\sup_{x\in\calGZero} \sum_{\sigma\in x\calG} |f(\sigma)|,
\sup_{x\in\calGZero} \sum_{\sigma\in\calG x} |f(\sigma)|
\Big\}.
\]

%==========================================================================================
\begin{prp}
\label{prp:regReprGroupoid}
%Let $\calG$ be a locally compact, Hausdorff, \'etale groupoid, 
Fix $x\in \calGZero$.
Let $f\in C_c(\calG)$ and let $\xi\in C_c(\calG x)$.
Then $f\ast\xi$ belongs to $C_c(\calG x)$, and
\[
\| f\ast\xi \|_p \leq \|f\|_I \|\xi\|_p
\]
for every $p\in[1,\infty]$.
It follows that there exists a unique contractive representation $\lambda_x\colon C_c(\calG)\to \Bdd(\ell^p(\calG x))$
satisfying $\lambda_x(f)(\xi)=f\ast \xi$ for $f\in C_c(\calG)$ and $\xi\in C_c(\calG x)$.
\end{prp}
\begin{proof}
Note that $f\ast\xi$ belongs to $C_c(\mathcal{G}x)$. 
Moreover, we have
\[
\| f\ast\xi \|_1 
= \sum_{\gamma\in\calG x} \Big| \sum_{\sigma\in\calG x}f(\gamma\sigma^{-1})\xi(\sigma) \Big|
\leq \sum_{\sigma\in\calG x} |\xi(\sigma)| \sum_{\gamma\in\calG x}  |f(\gamma\sigma^{-1})|
\leq \|f\|_I \|\xi\|_1,
\]
and
\begin{align*}
\| f\ast\xi \|_\infty
&= \sup_{\gamma\in\calG x} \Big| \sum_{\tau\in \ran(\gamma)\calG} f(\tau)\xi(\tau^{-1}\gamma) \Big| \\
&\leq \sup_{\gamma\in\calG x} \sum_{\tau\in \ran(\gamma)\calG} |f(\tau)| \sup_{\tau\in \ran(\gamma)\calG} |\xi(\tau^{-1}\gamma)|  
%\leq \sum_{\tau\in \ran(\gamma)\calG} |f(\tau)| \sup_{\gamma\in\calG x} |\xi(\tau^{-1}\gamma)|
\leq \|f\|_I \|\xi\|_\infty.
\end{align*}
Thus, the operator $C_c(\calG x)\to C_c(\calG x)$ given by left convolution by $f$ is bounded, and has norm at most $\|f\|_I$ for the $1$- and $\infty$-norm on $C_c(\calG x)$.
Hence, the norm inequality in the statement
follows from the Riesz-Thorin interpolation theorem.
The second assertion in the statement is immediate.
\end{proof}

%==========================================================================================
For $x\in \mathcal{G}^{(0)}$, we call the representation $\lambda_x$, constructed in the proposition above, the \emph{left regular representation} of $\calG$ associated to $x$. From now on, we fix $p\in [1,\infty)$.

%==========================================================================================
\begin{dfn}
\label{df:GroupoidLpOpAlg}
%Let $\calG$ be a locally compact, Hausdorff, \'etale groupoid, and let $p\in[1,\infty)$.
The \emph{reduced groupoid $L^p$-operator algebra} of $\calG$, denoted $F^p_\lambda(\calG)$, is the completion of $C_c(\calG)$ in the norm which for 
$f\in C_c(\calG)$ is given by
\[
\|f\|_\lambda := \sup \big\{ \|\lambda_x(f)\| \colon x\in\calGZero \big\}.
\]
%for all $f\in C_c(\calG)$. 
\end{dfn}

Although there is a potential conflict of notation 
with the 
norm introduced in \autoref{df:CrossedProds},
the notation $\|\cdot\|_\lambda$ is standard in 
the field. Since the norms from 
\autoref{df:GroupoidLpOpAlg} and \autoref{df:CrossedProds} are defined on different
objects, it should always be clear which one we refer
to.

%==========================================================================================
The definition above agrees with the one given in
Definition~6.12 of~\cite{GarLup17ReprGrpdLp}; 
see Corollary~6.15 in~\cite{GarLup17ReprGrpdLp} (observing
that what we here call $\lambda_x$ is denoted $\mathrm{Ind}(x)$ in \cite{GarLup17ReprGrpdLp}). 
Note that $\bigoplus_{x\in\calGZero}\lambda_x$ is an isometric representation 
of $F^p_\lambda(\calG)$ on an $L^p$-space, and thus $F^p_\lambda(\calG)$ is an $L^p$-operator algebra.
Moreover, $F^p_\lambda(\calG)$ is unital if $\calGZero$ is compact\footnote{The converse is likely to be true, but to the best of our knowledge this is not known.}. 

Let $p'\in (1,\infty]$ be the dual H\"older exponent of $p$,
which satisfies $\frac{1}{p}+\frac{1}{p'}=1$. For $x\in\calGZero$,
we identify $\ell^{p'}(\calG x)$ with the dual of $\ell^{p}(\calG x)$, and 
for $\xi\in \ell^p(\calG x)$ and $\eta\in \ell^{p'}(\calG x)$ we write $\langle \xi,\eta\rangle$ 
for the pairing given by 
$\langle \xi,\eta\rangle=\sum_{\gamma\in\mathcal{G}x}\xi(\gamma)\eta(\gamma)$.

%==========================================================================================
\begin{lma}
\label{prp:moveDelta}
Let $a\in F^p_\lambda(\calG)$, let $x\in\calGZero$, and let $\sigma,\gamma\in\calG x$.
Then
\[
\left\langle \lambda_x(a)(\delta_\sigma), \delta_\gamma \right\rangle
= \left\langle \lambda_{\ran(\sigma)}(a)(\delta_{\ran(\sigma)}), \delta_{\gamma\sigma^{-1}} \right\rangle.
\]
\end{lma}
\begin{proof}
Set $y:=\ran(\sigma)$.
To verify the formula for $C_c(\calG)$, let $f\in C_c(\calG)$.
Then
\[
\left\langle \lambda_x(f)(\delta_\sigma), \delta_\gamma \right\rangle
= \left\langle f\ast\delta_\sigma, \delta_\gamma \right\rangle
= (f\ast\delta_\sigma)(\gamma)
= f(\gamma\sigma^{-1})
= \left\langle \lambda_y(f)(\delta_y), \delta_{\gamma\sigma^{-1}} \right\rangle.
\]
For general elements in $F^p_\lambda(\calG)$, the formula follows since both sides of the equation are continuous in the norm of $F^p_\lambda(\calG)$.
\end{proof}

%==========================================================================================
\begin{lma}
\label{prp:comparisonNorms}
%Let $\calG$ be a locally compact, Hausdorff, \'etale groupoid, and let $p\in[1,\infty)$.
Given $f\in C_c(\calG)$, we have
\[
\|f\|_\infty \leq \|f\|_\lambda \leq \|f\|_I.
\]
Further, we have $\|f\|_\lambda=\|f\|_\infty$ for all $f\in C_c(\calGZero)$.
\end{lma}
\begin{proof}
Let $f\in C_c(\calG)$.
% The second inequality follows directly using that
Since
$\|\lambda_x(f)\|\leq\|f\|_I$ for all $x\in\calGZero$
by \autoref{prp:regReprGroupoid},
the second inequality follows.
To show the first inequality, let $\gamma\in\calG$.
We need to verify that $|f(\gamma)|\leq\|f\|_\lambda$.
Set $x:=\dom(\gamma)$.
Then
\[
\|f\|_\lambda
\geq \|\lambda_x(f)\|
\geq \|\lambda_x(f)(\delta_x)\|_p
= \big\| \sum_{\sigma\in\calG x} f(\sigma^{-1}) \big\|_p
\geq |f(\gamma)|.
\]

Lastly, if $f\in C_c(\calGZero)$, then $\|f\|_\infty=\|f\|_I$ and the statement follows.
\end{proof}

%==========================================================================================
\begin{ntn}
\label{ntn:mapJ}
By the first inequality in \autoref{prp:comparisonNorms}, 
the identity on $C_c(\calG)$ extends to a contractive
linear map $j\colon F^p_\lambda(\calG) \to C_0(\calG)$.
We abbreviate $j(a)$ to $j_a$ for $a\in F^p_\lambda(\calG)$.
\end{ntn}

%==========================================================================================
\begin{prp}
\label{prp:mapJ}
%Let $\calG$ be a locally compact, Hausdorff, \'etale groupoid, and let $p\in[1,\infty)$.
The map $j\colon F^p_\lambda(\calG) \to C_0(\calG)$ is injective and we have
\[
j_a(\gamma)
= \left\langle \lambda_{\dom(\gamma)}(a)(\delta_{\dom(\gamma)}), \delta_\gamma \right\rangle
\]
for every $a\in F^p_\lambda(\calG)$ and $\gamma\in\calG$.
\end{prp}
\begin{proof}
The proof is based on the proof of \cite[Proposition~3.3.3]{Sim17arX:EtaleGpds}.
To verify the displayed formula for $j$, let $f\in C_c(\calG)$ and let $\gamma\in\calG$.
Set $x:=\dom(\gamma)$.
Then 
\[
\left\langle \lambda_x(f)(\delta_{x}), \delta_\gamma \right\rangle
= \left\langle f\ast\delta_{x}, \delta_\gamma \right\rangle
= f(\gamma)
= j_f(\gamma),
\]
as desired.
Since both sides of the equation are continuous with respect to the norm in $F^p_\lambda(\calG)$, we obtain the same formula for all elements in $F^p_\lambda(\calG)$.

To show injectivity of $j$, let $a\in F^p_\lambda(\calG)$ with $a\neq 0$.
Choose $x\in\calGZero$ such that $\lambda_x(a)\neq 0$.
Then choose $\sigma,\tau\in\calG x$ such that $\langle \lambda_x(a)\delta_\sigma, \delta_\tau \rangle \neq 0$.
Set $y:=\ran(\sigma)$.
Using \autoref{prp:moveDelta} at the second step, we obtain
\[
j_a(\tau\sigma^{-1})
= \left\langle \lambda_y(a)(\delta_{y}), \delta_{\tau\sigma^{-1}} \right\rangle
= \langle \lambda_x(a)(\delta_\sigma), \delta_\tau \rangle \neq 0
\]
and thus $j$ is injective.
\end{proof}

%==========================================================================================

\begin{lma}
\label{prp:mapsLR}
%Let $\calG$ be a locally compact, Hausdorff, \'etale groupoid, let $p\in[1,\infty)$, and 
Let $p'\in(1,\infty]$ be the H\"older exponent that is dual to $p$, and let $x\in\calGZero$.
For $a\in F^p_\lambda(\calG)$, we write $\lambda_x(a)^{\prime}\colon\ell^{p}(\calG x)^{\prime}\to \ell^{p}(\calG x)^{\prime}$ for the transpose of $\lambda_x(a)$.
Define contractive linear maps $\ell_x\colon F^p_\lambda(\calG)\to\ell^p(\calG x)$ and $r_x\colon F^p_\lambda(\calG)\to\ell^{p'}(\calG x)$ by
\[
\ell_x(a) := \lambda_x(a)(\delta_x), \andSep
r_x(a) := \lambda_x(a)^{\prime}(\delta_x),
\]
for $a\in F^p_\lambda(\calG)$.
Then $\ell_x(a)(\gamma) = j_a(\gamma)$ and $r_x(a)(\gamma) = j_a(\gamma^{-1})$,
for all $a\in F^p_\lambda(\calG)$ and $\gamma\in\calG x$.
\end{lma}
\begin{proof}
Let $a\in F^p_\lambda(\calG)$ and $\gamma\in\calG x$.
Using \autoref{prp:mapJ} at the last step, we get
\[
\ell_x(a)(\gamma)
= \left\langle \lambda_x(a)(\delta_x),\delta_\gamma \right\rangle
= j_a(\gamma).
\]
Using \autoref{prp:moveDelta} at the third step
and \autoref{prp:mapJ} at the last one, we also get
\begin{align*}
r_x(a)(\gamma)
&= \left\langle \lambda_x(a)'(\delta_x),\delta_\gamma \right\rangle
= \left\langle \lambda_x(a)(\delta_\gamma),\delta_x \right\rangle \\
&= \left\langle \lambda_{\dom(\gamma)}(a)(\delta_{\dom(\gamma)}),\delta_{\gamma^{-1}} \right\rangle
= j_a(\gamma^{-1}).\qedhere
\end{align*}
\end{proof}

%==========================================================================================

%\begin{lma}
%\label{prp:multipliers}
%Let $\calG$ be a locally compact, Hausdorff, \'etale groupoid, let $p\in[1,\infty)$, and let $\gamma\in\calG$.
Given $f\in C_c(\calG)$ and $\gamma\in\calG$, it 
is easy to check that
$\|\delta_\gamma\ast f\|_\lambda \leq \|f\|_\lambda$
and
$\|f\ast\delta_\gamma\|_\lambda \leq \|f\|_\lambda$.
In particular, it follows that left and right convolution by $\delta_\gamma$ extend to contractive, linear maps $F^p_\lambda(\calG)\to F^p_\lambda(\calG)$.

%==========================================================================================
\begin{prp}
\label{prp:JConv}
%Let $\calG$ be a locally compact, Hausdorff, \'etale groupoid, and let $p\in[1,\infty)$.
Let $a,b\in F^p_\lambda(\calG)$ and $\gamma\in\calG$.
Set $x:=\dom(\gamma)$.
Then
\[
j_{a\ast b}(\gamma)
= \left\langle r_x(\delta_{\gamma^{-1}}\ast a), \ell_x(b) \right\rangle
= \sum_{\sigma\in\calG x} j_a(\gamma\sigma^{-1}) j_b(\sigma),
\]
and the sum is absolutely convergent.
\end{prp}
\begin{proof}
The proof is based on the proof of \cite[Proposition~III.4.2]{Ren80GrpdApproach}. Fix $\sigma\in \calG x$.

\textbf{Claim~1:}
\emph{We have $r_x(\delta_{\gamma^{-1}}\ast a)(\sigma)= j_a(\gamma\sigma^{-1})$.}
For $f\in C_c(\calG)$, we get
\begin{align*}
r_x(\delta_{\gamma^{-1}}\ast f)(\sigma)
&= \left\langle \lambda_x(\delta_{\gamma^{-1}}\ast f)'(\delta_x), \delta_\sigma \right\rangle 
= \left\langle \delta_x, \lambda_x(\delta_{\gamma^{-1}}\ast f)(\delta_\sigma) \right\rangle \\
&= \left\langle  \delta_x, \delta_{\gamma^{-1}}\ast f\ast \delta_\sigma \right\rangle
= f(\gamma\sigma^{-1}) = j_f(\gamma\sigma^{-1}).
\end{align*}
Now the claim follows since $C_c(\calG)$ is dense in $F^p_\lambda(\calG)$.

\textbf{Claim~2:}
\emph{We have $\lambda_x(a)'(\delta_\gamma)(\sigma) = j_a(\gamma\sigma^{-1})$.}
Using \autoref{prp:moveDelta} at the third step, we get
\begin{align*}
\lambda_x(a)'(\delta_\gamma)(\sigma)
&= \left\langle \lambda_x(a)'(\delta_\gamma), \delta_\sigma \right\rangle= \left\langle \delta_\gamma, \lambda_x(a)(\delta_\sigma) \right\rangle \\
&= \left\langle \lambda_{\ran(\sigma)}(a)(\delta_{\ran(\sigma)}), \delta_{\gamma\sigma^{-1}} \right\rangle=j_a(\gamma\sigma^{-1}).
\end{align*}

It follows from Claims~1 and~2 that $\lambda_x(a)'(\delta_\gamma)=r_x(\delta_{\gamma^{-1}}\ast a)$, and therefore
\begin{align*}
j_{a\ast b}(\gamma)
&= \left\langle \lambda_x(a)[\lambda_x(b)(\delta_x)], \delta_\gamma \right\rangle \\
&= \left\langle \lambda_x(a)'(\delta_\gamma), \lambda_x(b)(\delta_x) \right\rangle\\
&= \left\langle r_x(\delta_{\gamma^{-1}}\ast a), \ell_x(b) \right\rangle,
\end{align*}
which proves the first equality.
Now the second equality follows from Claim~1 and \autoref{prp:mapsLR}.
Moreover, the sum is absolutely convergent since it is given by the pairing between $\ell^p(\calG x)$ and its dual $\ell^{p'}(\calG x)$.
\end{proof}

%==========================================================================================
\begin{ntn}
%Let $\calG$ be a locally compact, Hausdorff, \'etale groupoid, and let $p\in[1,\infty)$.
The inclusion $C_c(\calGZero)\subseteq C_c(\calG)$ extends to an isometric, multiplicative map $C_0(\calGZero)\to F^p_\lambda(\calG)$, which we use to identify $C_0(\calGZero)$ with a closed subalgebra of $F^p_\lambda(\calG)$.
We let $E\colon F^p_\lambda(\calG) \to C_0(\calGZero)$ denote the composition of $j$ followed by the restriction $C_0(\calG)\to C_0(\calGZero)$.
\end{ntn}

%==========================================================================================
\begin{prp}
\label{prp:condExp}
%Let $\calG$ be a locally compact, Hausdorff, \'etale groupoid, and let $p\in[1,\infty)$.
The map $E$ defined above is contractive and satisfies 
$E(f)=f$ and $E(fag)=fE(a)g$ 
for all $f,g\in C_0(\calGZero)$ and $a\in F^p_\lambda(\calG)$.
In particular, if $\calGZero$ is compact, then $E$ is a conditional expectation in the sense of \autoref{df:CondExp}.
\end{prp}
\begin{proof}
We have $E(f)=f$ for every $f\in C_c(\calGZero)$, which implies the same for elements in $C_0(\calGZero)$.
Now, let $f,g\in C_0(\calGZero)$, let $a\in F^p_\lambda(\calG)$, and let $x\in\calGZero$.
Using \autoref{prp:JConv} at the second and fourth step, and using at the third and fifth step that $j_g(\sigma)=0$ and $j_f(\sigma^{-1})=0$ unless $\sigma=\dom(\sigma)=\ran(\sigma)$, we get
\begin{align*}
E(fag)(x)
&= j_{fag}(x)
= \sum_{\sigma\in\calG x} j_{fa}(\sigma^{-1}) j_g(\sigma)\\
&= j_{fa}(x) g(x) 
= \sum_{\sigma\in\calG x} j_f(\sigma^{-1}) j_a(\sigma) g(x)\\
&= f(x) j_a(x) g(x) 
= [fE(a)g](x),
\end{align*}
as desired. Since $E(1)=1$ when $\calGZero$ is compact, the last assertion follows.
\end{proof}

%==========================================================================================
%==========================================================================================
\section{\texorpdfstring{$L^p$}{Lp}-rigidity of reduced groupoid algebras}

%==========================================================================================
The main result of this section, \autoref{thm:GroupoidRigidity}, asserts that if $\calG$ is a topologically principal, Hausdorff, \'etale groupoid with compact unit space and $p\in [1,\infty)\setminus\{2\}$, then the Weyl groupoid of its reduced $L^p$-operator algebra is naturally isomorphic to $\calG$. 
This reveals a stark contrast with the case of \ca{s}, and further implications of this rigidity phenomenon will be given in Sections~\ref{sec:RigidDynSys} and~\ref{sec:tensCuntz}. 

It should be noted that virtually all concrete families of $L^p$-operator algebras that have been systematically studied so far can be realized as the $L^p$-operator algebras of \'etale groupoids;
see \cite{GarLup17ReprGrpdLp}.
Thus, adopting the groupoid perspective allows one to prove results about vast classes of algebras with a unified argument. 

%==========================================================================================
\begin{prp}
\label{prp:coreGpd}
Let $\calG$ be a Hausdorff, \'etale groupoid with compact unit space, and let $p\in [1,\infty)\setminus\{2\}$.
Then $\core(F^p_\lambda(\calG))=C(\calGZero)$.
\end{prp}
\begin{proof}
Let $a\in\core(F^p_\lambda(\calG))$. 
With $j\colon F^p_\lambda(\calG)\to C_0(\calG)$ defined as in \autoref{ntn:mapJ}, we will show that the support of $j_a$ is contained in~$\calGZero$. Once this is accomplished, it follows that $j_a$ belongs
to $C(\calGZero)$. Moreover, $a$ and $j_a$ are two elements in $F^p_\lambda(\calG)$ whose images under $j$ agree. Since $j$ is injective
by \autoref{prp:mapJ}, it follows that $a=j_a$ and hence $a$ belongs to 
$C(\calGZero)$.

Given $x\in\calGZero$, the homomorphism $\lambda_x\colon F^p_\lambda(\calG) \to \Bdd(\ell^p(\calG_x))$ from \autoref{prp:regReprGroupoid} 
is unital and contractive.
It follows from \autoref{prp:CorePreserved} (and \autoref{exa:BLp}) that $\lambda_x(a)$ is the multiplication operator in $\Bdd(\ell^p(\calG_x))$ given by some element in $\ell^\infty(\calG_x)$.
In particular, we have $\lambda_x(a)\delta_x = c \delta_x$ for some $c\in\CC$.

Let $\gamma\in\calG\setminus\calGZero$, so that 
$\gamma\neq \dom(\gamma)$.
Using this at the last step, and using \autoref{prp:mapJ} at the first step, we obtain (for some $c\in\CC$) that
\[
j_a(\gamma)
= \left\langle \lambda_{\dom(\gamma)}(a)\delta_{\dom(\gamma)}, \delta_\gamma \right\rangle
= \left\langle c \delta_{\dom(\gamma)}, \delta_\gamma \right\rangle
= 0. \qedhere
\]
\end{proof}

%==========================================================================================
%We always identify the $C^*$-core of $F^p_\lambda(\calG)$ with $C(\calGZero)$.
The computation of the following $C^*$-cores is an immediate consequence of \autoref{prp:coreGpd}.
We refer the reader to \cite{Phi13arX:LpCrProd} for the definition of the spatial $L^p$-UHF-algebras, and 
to~\cite{CorGui19LpOpAlgGraphs} for the definition of $L^p$-graph algebras.

%==========================================================================================
\begin{exas}
\label{exa:Cores}
Fix $p\in [1,\infty)\setminus\{2\}$.

(1)
For $n\in\NN$, we have $\core(M^p_n)=\CC^n$, identified as the diagonal matrices.
More generally, if $D$ is the $L^p$ UHF-algebra of type $2^{n_1}3^{n_3}\cdots q^{n_q}\cdots$,
then $\core(D)$ can be canonically identified with the continuous functions on the Cantor
space $\prod_{q \ \mathrm{ prime}} \prod_{j=1}^{n_q}\{1,\ldots,q\}$.
A similar description can be obtained for AF-algebras in terms of their Bratteli diagrams.

(2)
If $Q$ is a finite directed graph, then $\core(\mathcal{O}^p(Q))$ can be canonically 
identified with $\overline{\mathrm{span}}\{s_\mu s_\mu^*\colon \mu \mbox{ path in } Q\}$.
In particular, for the $L^p$-Cuntz algebra $\Onp$ (see \autoref{dfn:Ln}
and the comments after it), the spectrum of $\core(\Onp)$,
for $n\geq 2$, can be canonically identified with the Cantor space.
\end{exas}

%==========================================================================================
We now proceed to relate two classes of partial homeomorphisms on $\calGZero$:
the ones that are realized by admissible pairs in $F^p_\lambda(\calG)$
(\autoref{df:Normalizer}), and the ones that are induced by open bisections of $\calG$ (\autoref{rem:OpenBisGerms}).

Given a topological space $X$ and a continuous function $h\colon X\to\RR$, we use $\supp(h)$ to denote the open support of $h$, that is, $\supp(h)=\{x\in h:f(x)\neq 0\}$.
An open subset $U\subseteq X$ is called a \emph{cozero set} if there exists a continuous function $h\colon X\to\RR$ such that $U=\supp(h)$. If $X$ is normal (meaning
that disjoint closed subsets can be separated with disjoint open sets), then an open
set is cozero if and only if it is $F_\sigma$. In particular, every open set in a
compact, metric space is cozero. In general, every open set in a normal space 
contains a cozero open subset.

%==========================================================================================
\begin{prp}
\label{prop:partHomeoOfAdmisPair}
Let $\calG$ be a Hausdorff, \'etale groupoid with compact unit space, let $p\in [1,\infty)\setminus\{2\}$, and let $S$ be an open bisection of $\calG$ with associated partial homeomorphism $\beta_S\colon\dom(S)\to\ran(S)$.
Let $U\subseteq\calGZero$ be a cozero set.
Then the restriction of $\beta_S$ to $U$ is realizable by an admissible pair in $F^p_\lambda(\calG)$.
%Then the partial homeomorphism of $\calGZero$ induced by $S$ is realizable by an admissible pair in $F^p_\lambda(\calG)$.
\end{prp}
\begin{proof}
Replacing $S$ by $\{\gamma\in S : \dom(\gamma)\in U\}$, we may assume that $\dom(S)=U$.
Choose $h\in C(\calGZero)_+$ with $U=\supp(h)$ and define $a,b\colon\calG\to\CC$ by
\[
a(\gamma)=
\begin{cases*}
h(\dom(\gamma)), & if $\gamma\in S$ \\
0, & otherwise
\end{cases*}  
\ \ \mbox{ and} \ \ 
b(\gamma)=
\begin{cases*}
h(\ran(\gamma)), & if $\gamma^{-1}\in S$ \\
0, & otherwise
\end{cases*}
\]
for all $\gamma\in\calG$.
Then $a$ and $b$ belong to $F^p_\lambda(\calG)$ since they are $I$-norm limits of elements in $C_c(S)$.

We show that $s=(a,b)$ is an admissible pair that realizes $\beta_S\colon \dom(S)\to \ran(S)$.
To simplify the notation, we will omit the map $j$ from \autoref{ntn:mapJ} and identify elements in $F^p_\lambda(\calG)$ with their images in $C_0(\calG)$.
To check the first condition in
\autoref{df:Normalizer}, let $f\in C(\calGZero)_+$ and let $\gamma\in\calG$.
Then
\[
bfa(\gamma)
= \sum_{\sigma\in\calG\dom(\gamma)} b(\gamma\sigma^{-1})f(\ran(\sigma))a(\sigma).
\]
If $bfa(\gamma)\neq 0$, then there is $\sigma\in\calG$ with $b(\gamma\sigma^{-1})f(\ran(\sigma))a(\sigma)\neq 0$, which implies that $\sigma\in S$ and $\sigma\gamma^{-1}=(\gamma\sigma^{-1})^{-1}\in S$, and since $S$ is a bisection we get $\gamma^{-1}=\dom(\sigma)$, that is, $\gamma\in\calGZero$.
Thus, the support of $bfa$ is contained in $\calGZero$, and it follows that $bfa \in C(\calGZero)_+$, as desired.
Analogously, one obtains $afb\in C(\calGZero)_+$.

Now we check the second condition.
The first condition implies that $ba$ and $ab$ belong to $C(\calGZero)_+$.
Given $x\in\dom(S)$, let $\sigma_0\in S$ be the unique element satisfying $\dom(\sigma_0)=x$.
Then
\[
ba(x)
= \sum_{\sigma\in\calG x} b(\sigma^{-1})a(\sigma)
= b(\sigma_0^{-1})a(\sigma_0)
= h(x)^2.
\]
If $x\in\calGZero\setminus\dom(S)$, then $ba(x)=0$.
Thus $ba=h^2$, which implies that
\[
\dom(S) = \supp(h) 
= \supp(h^2)
= \supp(ba).
\]
Similarly one shows that $ab(\beta_S(x))=h(x)^2$ for $x\in\dom(S)$ and $ab(\beta_S(x))=0$ for $x\in\calGZero\setminus\dom(S)$, and thus
\[
\ran(S) = \beta_S(\dom(S))
= \supp(ab).
\]

To check the third condition, let $x\in U_s=\dom(S)$ and $f\in C_0(V_s)=C_0(\ran(S))$.
Let $\sigma_0\in S$ be the unique element with $\dom(\sigma_0)=x$.
Then $\ran(\sigma_0)=\beta_S(x)$, so
\begin{align*}
bfa(x)
&= \sum_{\sigma\in\calG x} b(\sigma^{-1})f(\ran(\sigma))a(\sigma)
= b(\sigma_0^{-1})f(\ran(\sigma_0))a(\sigma_0) \\
&= h(x)^2 f(\beta_S(x))
= f(\beta_S(x))ba(x).
\end{align*}
Analogously, one shows that $agb(y)=g(\beta_S^{-1}(y))ab(y)$ for 
$y\in V_s$ and $g\in C_0(U_s)$.
It follows that $s=(a,b)$ is an admissible pair that realizes $\beta_S$.
\end{proof}

%==========================================================================================
Next, we show that for a groupoid as in the previous lemma, which is moreover
\emph{topologically principal} (\autoref{df:topPpal}), any admissible pair naturally determines an open bisection such that the respective induced partial homeomorphisms on $\calGZero$ agree.

%==========================================================================================
\begin{prp}
\label{prop:partHomeoOfBisection}
Let $\calG$ be a topologically principal, Hausdorff, \'etale groupoid with compact unit space,
let $p\in [1,\infty)\setminus\{2\}$, and let $s=(a,b)$ be an admissible pair in 
$F^p_\lambda(\calG)$.
Set
\[
S := \big\{ \gamma\in\calG \colon a(\gamma), b(\gamma^{-1})\neq 0 \big\}.
\]
Then $S$ is an open bisection in $\calG$ such that $\beta_{S}=\alpha_s$.
\end{prp}
\begin{proof} 
To simplify the notation, we will omit the map $j$ from \autoref{ntn:mapJ} and identify elements in $F^p_\lambda(\calG)$ with their images in $C_0(\calG)$.

\textbf{Claim~1:}
\emph{Let $\gamma\in\calG$. Then $a(\gamma)b(\gamma^{-1})\geq 0$.}
Arguing by contradiction, assume that $a(\gamma)b(\gamma^{-1})<0$.
Since $a$ and $b$ are continuous (when viewed as functions on $\calG$), we can choose an open neighborhood $U$ of $\gamma$ such that $a(\sigma)b(\sigma^{-1})<0$ for all $\sigma\in U$.
Set $V:=\dom(U)$, which is an open subset of $\calGZero$.
Since $\calG$ is topologically principal, there is $x_0\in V$ with trivial isotropy.
Fix $\sigma_0\in U$ with $\dom(\sigma_0)=x_0$ and set $y_0:=\ran(\sigma_0)$.
Since $x_0$ has trivial isotropy, $\sigma_0$ is the unique element in $\calG x$ with range $y_0$.
Since
\[
ba(x)=\sum_{\sigma\in\calG x_0} b(\sigma^{-1})a(\sigma)
\]
converges absolutely, the set $\{\sigma\in\calG x_0\colon b(\sigma^{-1})a(\sigma)\neq 0\}$ is at most countable.
Set $t:=b(\sigma_0^{-1})a(\sigma_0)<0$.
Choose a neighborhood $W$ of $y_0$ in $\calGZero$ such that
\[
\sum_{\sigma\in\calG x_0, \ran(\sigma)\in W\setminus\{y_0\}} |b(\sigma^{-1})a(\sigma)| < |t| = -t.
\]
Choose $f\in C_0(\calGZero)$ with $0\leq f\leq 1$, with $f(y_0)=1$, and such that the support of $f$ is contained in $W$.
Then
\begin{align*}
bfa(x_0)
&= \sum_{\sigma\in\calG x_0} b(\sigma^{-1})f(\ran(\sigma))a(\sigma) \\
&= b(\sigma_0^{-1})f(\ran(\sigma_0))a(\sigma_0) + \sum_{\sigma\in\calG x_0, \ran(\sigma)\in W\setminus\{y_0\}} b(\sigma^{-1})f(\ran(\sigma))a(\sigma) \\
&\leq t + \sum_{\sigma\in\calG x_0, \ran(\sigma)\in W\setminus\{y_0\}} |b(\sigma^{-1})a(\sigma)|
< 0,
\end{align*}
which contradicts condition (1) in \autoref{df:Normalizer}.

\textbf{Claim~2:}
\emph{Let $\gamma\in S$. Then $\dom(\gamma)\in U_s$ and $\ran(\gamma)\in V_s$.}
Set $x:=\dom(\gamma)$ and $y:=\ran(\gamma)$.
Applying Claim~1 at the second step, we get
\[
ba(x) = \sum_{\sigma\in\calG x} b(\sigma^{-1})a(\sigma) \geq b(\gamma^{-1})a(\gamma) > 0,
\]
which by condition (2) in \autoref{df:Normalizer} implies that $x\in U_s$.
Analogously, we have
\[
ab(y) = \sum_{\sigma\in y\calG} a(\sigma)b(\sigma^{-1}) \geq a(\gamma)b(\gamma^{-1}) > 0,
\]
which implies that $y\in V_s$.

\textbf{Claim~3:}
\emph{Let $\gamma\in S$, and set $x:=\dom(\gamma)$. Then $\ran(\gamma)=\alpha_s(x)$.}
Assume that $\ran(\gamma)\neq \alpha_s(x)$.
Choose $f\in C_0(V_s)_+$ with $f(\alpha_s(x))=0$ and $f(\ran(\gamma))=1$.
Using the third condition in \autoref{df:Normalizer} at the second step, we get
\[
0 = f(\alpha_s(x))
= \frac{bfa(x)}{ba(x)}
= \sum_{\sigma\in\calG x} \frac{b(\sigma^{-1})f(\ran(\sigma))a(\sigma)}{ba(x)} 
\geq \frac{b(\gamma^{-1})a(\gamma)}{ba(x)} >0.
\]
This contradiction proves the claim.

\smallskip
Consider  the set
\[
T := \big\{ \eta\in\calG \colon \dom(\eta)\in U_s \mbox{ and } \ran(\eta)=\alpha_s(\dom(\eta)) \big\}.
\]
We have shown that $S\subseteq T$, and hence
\[
SS^{-1}\subseteq TT^{-1}\subseteq\calG' := \big\{ \gamma\in\calG \colon \ran(\gamma)=\dom(\gamma) \big\}.
\]
Denote by $\iota\colon\calG\to\calG$ the inversion
map, which is continuous.
Since $a$ and $b\circ\iota$ are continuous (as functions on $\calG$), their supports are open subsets, 
and hence $S=\supp(a)\cap \supp(b\circ\iota)$ is also open in $\calG$.
Therefore the open set $SS^{-1}$ is contained in the interior of $\calG'$, which equals $\calGZero$ since $\calG$ is topologically principal and thus effective;
see \cite[Proposition~3.6]{Ren08Cartan}.
An analogous reasoning implies that $S^{-1}S$ is contained in~$\calGZero$, and so $S$ is an open bisection.

\textbf{Claim~4:}
\emph{Let $x\in U_s$. Then there exists $\gamma\in S$ with $\dom(\gamma)=x$.}
We have
\[
0<ba(x)=\sum_{\sigma\in\calG x} b(\sigma^{-1})a(\sigma)
\]
which implies that there is $\gamma\in\calG x$ with $b(\gamma^{-1})a(\gamma)>0$.
Then $\dom(\gamma)=x$ and $\gamma\in S$, as desired.

\smallskip
It follows from Claim~4 that $\dom(S)=U_s$, and it remains to verify $\beta_S=\alpha_s$.
Let $x\in U_s$ and let $\gamma\in S$ be the unique element with $\dom(\gamma)=x$.
By Claim~3, we have
$\beta_S(x) = \ran(\gamma) = \alpha_s(x)$,
as desired.
\end{proof}

%==========================================================================================
The following is the main result of this section. 
It shows that a large class of groupoids can be recovered from their reduced groupoid $L^p$-operator algebras.

%==========================================================================================
\begin{thm}
\label{thm:GroupoidRigidity}
Let $\calG$ be a topologically principal, Hausdorff, \'etale groupoid with compact unit space,
and let $p\in [1,\infty)\setminus\{2\}$.
Then there is a natural identification
\[
\calG_{F^p_\lambda(\calG)}\cong\calG.
\]
\end{thm}
\begin{proof}
We identify $C(\calGZero)$ with the $C^*$-core of $F^p_\lambda(\calG)$ as in \autoref{prp:coreGpd}.
Let $\mathcal{A}$ be the set of partial homeomorphisms of $\calGZero$ realized by admissible pairs in $F^p_\lambda(\calG)$. 
Further, let $\mathcal{B}$ be the family of partial homeomorphisms of $\calGZero$ induced by open bisections of $\calG$.

By \autoref{prop:partHomeoOfBisection}, we have $\mathcal{A}\subseteq\mathcal{B}$.
Applying \autoref{prop:partHomeoOfAdmisPair}, the converse inclusion holds if every open subset of $\calGZero$ is a cozero set (for example, if $\calGZero$ is metrizable), and in general it holds \emph{locally}, that is, for every $\beta\in\mathcal{B}$ and $x\in\calGZero$ there exists an open neighborhood~$U$ of $x$ such that $\beta|_U\in\mathcal{A}$.
It follows that the groupoids of germs of $\mathcal{A}$ and $\mathcal{B}$ are naturally isomorphic.

By definition, $\calG_{F^p_\lambda(\calG)}$ is the groupoid of germs of $\mathcal{A}$.
Since $\calG$ is isomorphic to the groupoid of germs of $\mathcal{B}$ (by \autoref{rem:OpenBisGerms}), the result follows.
\end{proof}

%==========================================================================================
\begin{cor}
Let $\calG$ and $\mathcal{H}$ be topologically principal, Hausdorff, \'etale groupoids with compact unit spaces, 
and let $p\in [1,\infty)\setminus\{2\}$. 
Then there is an isometric isomorphism $F^p_\lambda(\calG)\cong F^p_\lambda(\mathcal{H})$ 
if and only if there is a groupoid isomorphism $\calG\cong \mathcal{H}$.
\end{cor}

%==========================================================================================
%==========================================================================================
\section{\texorpdfstring{$L^p$}{Lp}-rigidity of dynamical systems}
\label{sec:RigidDynSys}

%==========================================================================================
We now specialize to transformation groupoids. % obtained from topological dynamical systems. 

%==========================================================================================
\begin{pgr}
\label{pgr:MapCrProdToTransfGrpd}
Let $G$ be a discrete group, and let $X$ be a compact Hausdorff space.
Given an action $G\curvearrowright X$ of $G$ on $X$, written $(g,x)\mapsto g\cdot x$, the \emph{transformation groupoid} $G\ltimes X$ is defined as the product space $G\times X$ endowed with the operations
\[
(g,h\cdot x)(h,x)=(gh,x), \andSep 
(g,x)^{-1}=(g^{-1},g\cdot x)
\]
for all $g,h\in G$ and all $x\in X$. 

The groupoid $G\ltimes X$ is \'etale when equipped with the natural product topology, and its unit space is $\{1\}\times X$, which we identify with $X$.
This also identifies $C(X)$ with a subalgebra of $C_c(G\ltimes X)$.
Let $\alpha\colon G\to\Aut(C(X))$ denote the associated action, given by $\alpha_g(f)(x)=f(g^{-1}x)$ for $g\in G$, $f\in C(X)$ and $x\in X$.
For $g\in G$, we let $u_g$ denote the characteristic function of $\{g\}\times X$ in $C_c(G\ltimes X)$.
Then $u_g f u_{g^{-1}} = \alpha_g(f)$ for $g\in G$ and $f\in C(X)$, with the product in $C_c(G\ltimes X)$ as in \autoref{dfn:convolution}.

Let $\varphi\colon C_c(G,C(X)) \to C_c(G\ltimes X)$ denote the map given by
\begin{align*}
\varphi(fu_g)(s,x)=\begin{cases}
             f(sx), & \text{ if } s=g;\\
             0, & \text{ if } s\neq g,
            \end{cases}
\end{align*}
for $f\in C(X)$ and $g\in G$. It follows from the above discussion
that $\varphi$ is an algebra isomorphism, where 
$C_c(G\ltimes X)$ is given the algebra structure from \autoref{pgr:crProduct}.

We let $c_G$ denote the counting measure on $G$.
Given a Borel probability measure $\mu$ on $X$, we use the notation $\mathrm{Ind}(\mu)\colon C_c(G\ltimes X)\to \Bdd(L^p(G\times X,c_G\times \mu))$ from Section~6.3 of \cite{GarLup17ReprGrpdLp}, and recall that $\mathrm{Ind}(\mu)$ is induced by convolution in $C_c(G\ltimes X)$.
More specifically, we have $\mathrm{Ind}(\mu)(a)\xi=a\ast\xi$ for all $a\in C_c(G\ltimes X)$ and all $\xi\in C_c(G\ltimes X)\subseteq L^p(G\times X,c_G\times \mu)$.
\end{pgr}

%==========================================================================================
\begin{lma}
\label{prp:computationRhoMu}
Let $G\curvearrowright X$ be an action of a discrete group $G$ on a compact, Hausdorff space $X$.
Let $\mu$ be a Borel probability measure on $X$, let $\pi_{\mu,0}\colon C(X)\to\Bdd(L^p(X,\mu))$ be the associated unital representation by multiplication operators, and let $\pi_{\mu}\rtimes \lambda^\mu$
be the induced regular representation as in \autoref{pgr:crProduct}.
Let $\varphi$ be the natural identification described in \autoref{pgr:MapCrProdToTransfGrpd}.

Then, using the natural identification $\ell^p(G,L^p(X,\mu)) \cong L^p(G\times X,c_G\times\mu)$, we have $\pi_\mu\rtimes \lambda^\mu = \mathrm{Ind}(\mu)\circ\varphi$.
This means that the following diagram commutes:
\[
\xymatrix{
C_c(G,C(X)) \ar[rr]^-{\pi_\mu\rtimes \lambda^\mu} \ar[d]_{\varphi}
&&  \Bdd(\ell^p(G,L^p(X,\mu))) \ar[d]^{\cong} \\
C_c(G\ltimes X) \ar[rr]_-{\mathrm{Ind}(\mu)}
&& \Bdd(L^p(G\times X,c_g\times\mu)).
}
\]
\end{lma}
\begin{proof}
Recall (see \autoref{pgr:crProduct}) that
$\pi_\mu\colon C(X)\to \Bdd(\ell^p(G,L^p(\mu))$ is given by
\[
\pi_\mu(f)(\xi)(g)=\pi_{\mu,0}(\alpha_{g^{-1}}(f))(\xi(g))
= \alpha_{g^{-1}}(f)\xi(g)
\]
for all $f\in C(X)$, all $\xi\in \ell^p(G,L^p(\mu))$ and all $g\in G$.
%, and denote by $(\pi_\mu, \lambda^\mu)$ the 
%associated covariant regular representation.
%We let $\rho_\mu:=\pi_\mu\rtimes \lambda^\mu\colon 
%C_c(G\times X)\to\Bdd(L^p(c_G\times\mu))$ be the induced regular representation.
%We write $\rho_\mu$ for $\pi_\mu\rtimes \lambda^\mu$.
By linearity, it is enough to verify $(\pi_\mu\rtimes \lambda^\mu)(fu_g)=\mathrm{Ind}(\mu)(\varphi(fu_g))$ for $f\in C(X)$ and $g\in G$. 
In this case, we have
\begin{align*}
(\pi_\mu\rtimes \lambda^\mu)(fu_g)(\xi)(s,x)
&= \pi_\mu(f)[\lambda^\mu_g(\xi)](s,x) \\
&= \pi_{\mu,0}(\alpha_s^{-1}(f))(x)\cdot \lambda^\mu_g(\xi)(s,x) 
= f(sx)\xi(g^{-1}s,x)
\end{align*}
for $s\in G$ and $x\in X$. 
On the other hand,
\begin{align*}
\mathrm{Ind}(\mu)(\varphi(fu_g))(\xi)(s,x)
&= (\varphi(fu_g)\ast\xi)(s,x) \\
&= \sum_{(t,y)\in G\ltimes X(1,x)} \varphi(fu_g)(st^{-1},ty)\xi(t,y)\\
&= \sum_{t\in G} \varphi(fu_g)(st^{-1},tx)\xi(t,x) 
= f(sx)\xi(g^{-1}s,x),
\end{align*}
as desired.
\end{proof}   

%==========================================================================================
It will be convenient for us to know that, when considering all regular covariant
representations of a dynamical system $G\curvearrowright X$, it suffices to 
consider representations of $C(X)$ on $L^p(X,\mu)$ by multiplication operators,
for Borel measures $\mu$ on $X$. The following lemma is the case of probability 
measures, while the general case can be reduced to it by considering separable
subsystems (see the proof of \autoref{prp:crProdVsTransfGrpdReduced}).

%==========================================================================================
\begin{lma}
\label{prp:desintigrate}
Let $G\curvearrowright X$ be an action of a discrete group on a compact, Hausdorff space.
Let $(Y,\nu)$ be a standard Borel probability space.
Let $\pi_0\colon C(X)\to \Bdd(L^p(Y,\nu))$ be a unital representation, and let $\pi\rtimes\lambda^\nu$ be the induced regular representation as in \autoref{pgr:crProduct}.
Let $\varphi\colon C_c(G,C(X)) \to C_c(G\ltimes X)$ be the natural identification described in \autoref{pgr:MapCrProdToTransfGrpd}.
Then
\[
\| \pi\rtimes\lambda^\nu(a) \| 
\leq \| \varphi(a) \|_{F^p_\lambda(G\ltimes X)}
%= \| \pi_{\mu}\rtimes\lambda^\mu(a) \|
\]
for every $a\in C_c(G,C(X))$.
\end{lma}
\begin{proof}
By \autoref{prp:CorePreserved}, the range of $\pi_0$ is contained in the C*-core of 
$\Bdd(L^p(\nu))$, which by \autoref{exa:BLp} coincides with the algebra of multiplication
operators by functions in $L^\I(\nu)$. 
We thus regard $\pi_0$ as a unital, contractive *-homomorphism $C(X)\to L^\I(\nu)$. 
Integration against $\nu$ defines a tracial state $\tau_\nu\colon L^\infty(\nu)\to \CC$. 
Hence, $\tau_\nu\circ\pi_0\colon C(X)\to \CC$ is also a tracial state, and thus there is a unique Borel probability measure $\mu$ on $X$ such that $\tau_\mu=\tau_\nu\circ\pi_0$.
To lighten the notation, write
\[
\rho
= \pi\rtimes \lambda^\nu
\colon C_c(G, C(X)) \to \Bdd\big( \ell^p(G,L^p(Y,\nu)) \big)
\] 
for the induced regular representation of $\pi_0$, and similarly write
\[
\rho_\mu 
= \pi_{\mu}\rtimes\lambda^\mu
\colon C_c(G,C(X))\to\Bdd\big( \ell^p(G,L^p(Y,\mu)) \big)
\]
for the induced regular representation of $\pi_\mu$.
Fix $a\in C(X)$ for the rest of this proof. 
Using \autoref{prp:desintigrate} at the first step and 
Corollary~6.15 in~\cite{GarLup17ReprGrpdLp} at the last step, we get
\[
\| \rho_\mu(a) \|
= \| \mathrm{Ind}(\mu)(\varphi(a)) \|
\leq \sup_{\mu'\in M_1(X)} \| \mathrm{Ind}(\mu')(\varphi(a)) \|
= \| \varphi(a) \|_{F^p_\lambda(G\ltimes X)}.
\]
It thus suffices to show that $\|\rho(a)\| = \|\rho_\mu(a)\|$.

Denote by $\overline{\pi}_0\colon L^\I(X,\mu)\to L^\I(Y,\nu)$ the unique extension of $\pi_0$ to a unital, injective, normal and trace-preserving *-homomorphism.
Then $\overline{\pi}_0$ induces\footnote{Since we
could not find a reference for this folklore result, we sketch the argument. Denote by $(\mathcal{A},\mu)$ the measure algebra of $\mu$, namely the quotient of 
the $\sigma$-algebra of $\mu$ by the ideal $\mathcal{N}$ of sets
of $\mu$-measure zero; and similarly for $(\mathcal{B},\nu)$. It is well-known that the set of projections
in $L^\infty(X,\mu)$ is isomorphic to $\mathcal{A}$ via identifying a class $e=E+\mathcal{N}\in \mathcal{A}$ with the indicator function $\chi_E$; this correspondence identifies $\tau_\mu$ with $\mu$.
We define $\widetilde{\kappa}\colon (\mathcal{A},\mu)\to (\mathcal{B},\nu)$ by letting 
$\widetilde{\kappa}(e) \in \mathcal{B}$ be the class associated to the projection $\varphi(\chi_E)\in L^\infty(Y,\nu)$, for $e=E+\mathcal{N}\in\mathcal{A}$. Finally, \cite[Theorem 343B]{Fre-MsrThy3b} implies that $\widetilde{\kappa}$ can be lifted to a measurable map $\kappa\colon Y\to X$,
which is then immediately seen to be measure-preserving (because so is $\widetilde{\kappa}$)
and essentially surjective (because $\widetilde{\kappa}$ is injective).}
a measurable, measure-preserving, essentially surjective map 
$\kappa\colon (Y,\nu)\to (X,\mu)$ such that $\overline{\pi}_0(f)=f\circ\kappa$ for every $f\in L^\infty(X,\mu)$.

Under the map $\kappa$, we regard $Y$ as fibered over $X$, and for $x\in X$ we write $Y_x$ for the standard Borel space $Y_x=\kappa^{-1}(x)$. 
The disintegration theorem (see, for example, the first paragraph on page~316 of \cite{Ram82TopMsrdGrps}) gives, for every $x\in X$, a Borel probability measure $\nu_x$ on $Y$, whose support is contained in $Y_x$, satisfying:
\begin{itemize}
\item 
for every Borel subset $B\subseteq Y$, the assignment $x\mapsto \nu_x(B)$ is 
Borel;
\item 
for every Borel function $f\colon Y\to \mathbb{R}$, we have
\[\int_Y f\ d\nu = \int_X \int_{Y_x} f\ d\nu_xd\mu.\]
\end{itemize}

Following Ramsay's terminology at the top of page 338 of \cite{Ram82TopMsrdGrps}, for $n\in \{1,2,\ldots,\infty\}$ we say that a probability space is of \emph{type $n$} if it is atomic and has exactly $n$ atoms.
For $n=0,-1,-2,\ldots,-\infty$, we say that a probability space is of \emph{type $n$} if it is not atomic and has exactly $-n$ atoms. 
It is well-known that the type of a standard Borel probability spaces determines its isomorphism class.
Write $\overline{\mathbb{Z}}$ for $\mathbb{Z}\cup \{-\infty,\infty\}$. 
Given $n\in\overline{\mathbb{Z}}$, we fix a standard Borel probability space $(Z_n,\delta_n)$ of type $n$.

For $n\in\overline{\mathbb{Z}}$, set 
\[
X_n = \big\{ x\in X : (Y_x,\nu_x) \mbox{ is of type } n \big\} \andSep
Y_n = \kappa^{-1}(X_n) \subseteq Y,
\]
and write $\mu_n$ for the restriction of $\mu$ to $X_n$, and similarly write $\nu_n$ for the restriction of $\nu$ to $Y_n$.
By Lemma~6.4 in~\cite{Ram82TopMsrdGrps}, $X_n$ is a Borel subset of $X$, and hence so is $Y_n\subseteq Y$, for each $n\in\overline{\ZZ}$.
Further, $X$ decomposes as a disjoint union $X=\bigsqcup_{n\in\overline{\mathbb{Z}}}X_n$, and similarly $Y=\bigsqcup_{n\in\overline{\mathbb{Z}}}Y_n$.

With respect to the natural decomposition $L^p(Y,\nu) \cong \bigoplus_{n\in\overline{\mathbb{Z}}} L^p(Y_n,\nu_n)$, every operator $\pi_0(f)\in\Bdd(L^p(Y,\nu))$ for $f\in C(X)$ is block-diagonal.
For the induced decomposition
\[
\ell^p(G,L^p(Y,\nu))
\cong \bigoplus_{n\in\overline{\mathbb{Z}}} \ell^p(G,L^p(Y_n,\nu_n)),
\]
it follows that 
$\rho(a)$ is block-diagonal, with blocks $\rho(a)_n\in\Bdd( \ell^p(G,L^p(Y_n,\nu_n)) )$
for $n\in\overline{\mathbb{Z}}$.
Similarly, with respect to the natural decomposition
\[
\ell^p(G,L^p(X,\mu))
\cong \bigoplus_{n\in\overline{\mathbb{Z}}} \ell^p(G,L^p(X_n,\mu_n))
\]
the operator $\rho_\mu(a)$ is block-diagonal, with blocks $\rho_\mu(a)_n\in \Bdd( \ell^p(G,L^p(X_n,\mu_n)) )$ for $n\in\overline{\mathbb{Z}}$. In particular, we have
\[
\|\rho(a)\| 
= \sup_{n\in\overline{\mathbb{Z}}} \|\rho(a)_n\| \ \ \mbox{ and } \ \ 
\| \rho_\mu(a) \|= \sup_{n\in\overline{\mathbb{Z}}} \|\rho_\mu(a)_n\|.
\]

Fix $n\in\overline{\mathbb{Z}}$.
Since $(Y_x,\nu_x)$ is isomorphic to $(Z_n,\delta_n)$ for all $x\in X_n$, the existence of a measurable section for 
$\kappa$ implies that
there is an isomorphism $(Y_n,\nu_n)\cong(X_n,\mu_n)\times(Z_n,\delta_n)$, such that the restriction $\kappa|_{Y_n}\colon Y_n \to X_n$ is identified with the projection onto the first coordinate.
%Choose a measurable family $(\psi_x)_{x\in X}$ of 
%isomorphisms $\psi_x\colon (Y_x,\nu_x)\to (Z,\delta)$.
%Then define $\rho\colon (Y_n,\nu_n)\to (X_n,\mu_n)\times 
%(Z,\delta)$ by $\rho(y)=(\kappa(y),\psi_{\kappa(y)}(y))$
%for $y\in Y$.)
We thus obtain an isometric isomorphism
\[
%L^p(Y_n,\nu_n) \cong L^p(X_n,\mu_n)\otimes^p L^p(Z_n,\delta_n), \andSep
\ell^p(G,L^p(Y_n,\nu_n))
\cong \ell^p(G,L^p(X_n,\mu_n))\otimes^p L^p(Z_n,\delta_n)
\]
that identifies $\rho(a)_n$ with $\rho_\mu(a)_n\otimes\id_{L^p(Z_n,\delta_n)}$.
It follows that
\[
\| \rho(a)_n \|
= \| \rho_\mu(a)_n\otimes\id_{L^p(Z_n,\delta_n)} \|
= \| \rho_\mu(a)_n\|.
\]

Using that that this holds for every $n\in\overline{\mathbb{Z}}$, we conclude that
\[
\|\rho(a)\| 
= \sup_{n\in\overline{\mathbb{Z}}} \|\rho(a)_n\|
= \sup_{n\in\overline{\mathbb{Z}}} \|\rho_\mu(a)_n\|
= \| \rho_\mu(a) \|.\qedhere
\]
\end{proof}

%==========================================================================================
\begin{prp}
\label{prp:crProdVsTransfGrpdReduced}
Let $p\in [1,\infty)$, and let $G\curvearrowright X$ be an action of a discrete group on a compact, Hausdorff space.
Then the natural identification $\varphi\colon C_c(G,C(X)) \to C_c(G\ltimes X)$ described in \autoref{pgr:MapCrProdToTransfGrpd} extends to an isometric isomorphism
\[
F^p_\lambda(G,C(X)) \cong F^p_\lambda(G\ltimes X).
\]
\end{prp}
\begin{proof}
Fix $a\in C_c(G,C(X))$ We
will show that $\|a\|_{F^p_\lambda(G,C(X))}=\|\varphi(a)\|_{F^p_\lambda(G\ltimes X)}$ , 
starting with the inequality `$\geq$'.

Let $M_1(X)$ denote the space of all Borel probability measures on $X$, and let $\mathcal{R}(G,X)$
denote the set of all integrated forms of regular covariant representations of $C_c(G,C(X))$ on $L^p$-spaces, in the sense of the discussion before \autoref{df:CrossedProds}. 
Given $\mu\in M_1(X)$, let $\pi_{\mu,0}\colon C(X)\to\Bdd(L^p(X,\mu))$ be the associated unital representation by multiplication operators, and let $\rho_\mu:=\pi_{\mu}\rtimes \lambda^\mu$ 
be the induced regular representation as in \autoref{pgr:crProduct}.
Note that $\rho_\mu$ belongs to $\mathcal{R}(G,X)$. % for all $\mu\in M_1(X)$.

%Let $\varphi\colon C_c(G,C(X)) \to C_c(G\ltimes X)$ be the natural identification described in \autoref{pgr:MapCrProdToTransfGrpd}.
Using Corollary~6.15 in~\cite{GarLup17ReprGrpdLp} at the first step (see also the second paragraph after \autoref{df:GroupoidLpOpAlg}), 
and using \autoref{prp:computationRhoMu} at the second step, we obtain
\begin{align*}
\label{eqn:Ineq1}\tag{6.1}
\|\varphi(a)\|_{F^p_\lambda(G\ltimes X)} &= \sup_{\mu\in M_1(X)} \|\mathrm{Ind}(\mu)(\varphi(a))\|
=\sup_{\mu\in M_1(X)} \|\rho_\mu(a)\|\\
&\leq
\sup_{\rho\in \mathcal{R}(G,X)}\|\rho(a)\|
= \|a\|_{F^p_\lambda(G,C(X))}.
\end{align*}
%for every $a\in C_c(G,C(X))$.
 
We now turn to the converse inequality `$\leq$'.
Let $\varepsilon>0$.
Using the definition of the norm on $F^p_\lambda(G,C(X))$, choose a measure space $(Y,\nu)$, a unital 
representation $\pi_0\colon C(X)\to \Bdd(L^p(Y,\nu))$, 
and $\xi\in\ell^p(G,L^p(Y,\nu))$ with $\|\xi\|_p=1$
such that,
with $\rho = \pi\rtimes\lambda^\nu\colon C_c(G\times  X)\to\Bdd(\ell^p(G,L^p(Y,\nu))$
denoting the induced regular representation as in \autoref{pgr:crProduct},
we have
\[
\|a\|_{F^p_\lambda(G,C(X))}-\varepsilon
< \|\rho(a)\xi\|_p.
\]

For $g\in G$, set $\xi_g=\xi(g)\in L^p(Y,\nu)$ and set
$a_g=a(g)\in C(X)$. 
Then at most countably many $\xi_g$ are nonzero (because $\|\xi\|^p_p=\sum_{g\in G}\|\xi\|_p^p<\I$), 
and at most finitely many $a_g$ are nonzero (because the support
of $a$ is finite). Let $G'$ denote the (countable) subgroup of 
$G$ generated by $\supp(\xi)$. Denoting by 
$\alpha\colon G\to\Aut(C(X))$ the induced action, 
it follows that the set
\[\{1_{C(X)}\}\cup \{\alpha_g(a_h)\colon g\in G', h\in G\}\]
is a countable and $G'$-invariant subset of $C(X)$.
Denote by $X'$ the spectrum of the C*-algebra it generates.
Then $G'$ acts on $X'$, and the canonical 
quotient map $X\to X'$ is $G'$-equivariant.
By construction, $a$ belongs to $C_c(G',C(X'))\subseteq C_c(G,C(X))$.

Choose a separable $L^p$-space $L^p(Y',\nu')\subseteq L^p(Y,\nu)$ that contains $\{\xi_g\colon g\in G\}$ and satisfies 
$\pi_0(b)\eta\in L^p(Y',\nu')$ for all $\eta\in L^p(Y',\nu')$;
see Proposition~1.25 (and its proof) in \cite{Phi13arX:LpCrProd}.
By construction, $\xi$ belongs to $\ell^p(G',L^p(Y',\nu'))\subseteq\ell^p(G,L^p(Y,\nu))$.

It is well-known that every separable $L^p$-space can be realized by a $\sigma$-finite measure, and thus by a probability measure;
see for example the corollary to Theorem~3 in Section~15 of \cite{Lac74IsoThyClassicalBSp}.
Thus, we may assume that $\nu'$ is a probability measure and that $Y'$ is a standard Borel space, which will allow us to apply \autoref{prp:desintigrate}.

Now $\pi_0$ induces a unital representation $\pi_0'\colon C(X')\to \Bdd(L^p(Y',\nu'))$.
Let
\[
\rho' = \pi'\rtimes\lambda^{\nu'}\colon C_c(G',C(X')) \to \Bdd(\ell^p(G',L^p(Y',\nu')))
\] 
denote the induced regular representation with respect to $G'\curvearrowright X'$.
We have $\rho(a)\xi=\rho'(a)\xi$ by construction.
Let 
\[\varphi'\colon C_c(G',C(X')) \to C_c(G'\ltimes X') \ \ \mbox{ and } \ \ \varphi\colon C_c(G,C(X)) \to C_c(G\ltimes X)\] denote the natural identifications described in \autoref{pgr:MapCrProdToTransfGrpd}.
Using \autoref{prp:desintigrate} at the last step, we get
\[
\|a\|_{F^p_\lambda(G,C(X))}-\varepsilon
< \|\rho(a)\xi\|_p
= \|\rho'(a)\xi\|_p
\leq \|\rho'(a)\|
\leq \|\varphi'(a)\|_{F^p_\lambda(G'\ltimes X')},
\]
Since by definition, we have
\[
\|\varphi'(a)\|_{F^p_\lambda(G'\ltimes X')}
= \sup_{x'\in X'} \| \lambda_{x'}(\varphi'(a)) \|,
\]
we may find $x'\in X'$ such that
\[
\|a\|_{F^p_\lambda(G',C(X'))}-\varepsilon
< \| \lambda_{x'}(\varphi'(a)) \|.
\]

Choose a preimage $x\in X$ of $x'$ under the quotient map $X\to X'$.
One verifies that $\| \lambda_{x'}(\varphi'(a)) \| =  \| \lambda_{x}(\varphi(a)) \|$, and consequently
\begin{align*}
\|a\|_{F^p_\lambda(G,C(X))}-\varepsilon
&< \| \lambda_{x'}(\varphi'(a)) \|.
= \| \lambda_{x}(\varphi(a)) \|\\
&\leq \sup_{y\in X} \| \lambda_{y}(\varphi(a)) \|
= \|\varphi(a)\|_{F^p_\lambda(G\ltimes X)}. 
\end{align*}
Since $\varepsilon>0$ is arbitrary, we conclude that 
$\|a\|_{F^p_\lambda(G,C(X))}
\leq \|\varphi(a)\|_{F^p_\lambda(G\ltimes X)}$.
\end{proof}

%==========================================================================================
\begin{prp}
\label{prp:crProdVsTransfGrpdFull}
Let $p\in [1,\infty)$, and let $G\curvearrowright X$ be an action of a discrete group on a compact, Hausdorff space.
Then the natural identification $\varphi\colon C_c(G,C(X)) \to C_c(G\ltimes X)$ extends to a unital, contractive homomorphism
\[
F^p(G,C(X)) \to F^p(G\ltimes X).
\]
\end{prp}
\begin{proof}
%We now consider the full crossed product and groupoid algebra.
As in \autoref{pgr:crProduct}, we denote the elements in $C_c(G,C(X))$ by finite linear combinations $\sum_{g\in G} a_gu_g$ with $a_g\in C(X)$.
Given a $L^p$-space $E$, it is straightforward to check that a unital homomorphism $\pi\colon C_c(G,C(X))\to\Bdd(E)$ is induced from a covariant representation of $(G,C(X))$ if and only if $\|\pi(\sum_{g\in G} a_g u_g)\|\leq\sum_{g\in G} \|a_g\|_\infty$ for every $\sum_{g\in G} a_g u_g\in C_c(G,C(X))$.
We have
\[
\Big\|\sum_{g\in G} a_g u_g \Big\|_I
= \max \Big\{ 
\sup_{x\in X} \sum_{g\in G} |a_g(x)|,
\sup_{x\in X} \sum_{g\in G} |a_g(g^{-1}x)|
\Big\}
\leq
\sum_{g\in G} \|a_g\|_\infty
\]
for all $\sum_{g\in G} a_g u_g\in C_c(G,C(X))$.
We conclude that every $I$-norm contractive, unital representation of $C_c(G,C(X))$ is induced from a covariant representation.
\end{proof}

%==========================================================================================
In \autoref{prp:crProdVsTransfGrpdFull},
we do not claim that 
the map $F^p(G,C(X),\alpha) \to F^p(G\ltimes X)$ is isometric, since we do not
know that a covariant representation is contractive with respect to the $I$-norm.
This is also a delicate point in the \ca{} setting; 
see \cite[Lemma~3.2.3]{Sim17arX:EtaleGpds}.

%==========================================================================================
Recall that an action of a discrete group $G$ on a compact Hausdorff space
$X$ is said to be \emph{topologically free} if $\{x\in X\colon g\cdot x=x \mbox{ implies } g=1\}$ 
is dense in~$X$.
Equivalently, the transformation groupoid $G\ltimes X$ is topologically principal;
see for example \cite[Section~4.2]{Sim17arX:EtaleGpds}.
For such transformation groupoids, groupoid isomorphism can be 
rephrased in terms of the underlying dynamics, via the notion of continuous orbit
equivalence, which we recall below (see Definition~2.5 in~\cite{Li18CtsOrbitEquiv}).

%==========================================================================================
\begin{dfn}
\label{df:coe}
Let $G$ and $H$ be discrete groups, let $X$ and $Y$ be compact Hausdorff spaces, 
and let $G\curvearrowright^\sigma X$ and $H\curvearrowright^\rho Y$ be actions.
One says that $\sigma$ and $\rho$ are \emph{continuously orbit equivalent} 
if there exist a homeomorphism $\theta\colon X\to Y$ and continuous cocycle maps
$c_H\colon G\times X\to H$ and $c_G\colon H\times Y\to G$ satisfying
\[
\theta(\sigma_g(x))=\rho_{c_H(g,x)}(\theta(x)), \andSep
\theta^{-1}(\rho_h(y))=\sigma_{c_G(h,y)}(\theta^{-1}(y))
\]
for all $x\in X$, $y\in Y$, $g\in G$ and $h\in H$.
\end{dfn}

%==========================================================================================
We are now ready to present our main application to isomorphisms of $L^p$-crossed products by topologically free actions.

%==========================================================================================
\begin{thm}
\label{thm:RigidityDynSysts}
Let $p\in [1,\infty)\setminus\{2\}$, let $G$ and $H$ be discrete groups, 
let $X$ and $Y$ be compact Hausdorff spaces, and let $G\curvearrowright X$ and $H\curvearrowright Y$ be topologically free actions.
Then the following are equivalent:
\begin{enumerate}
\item 
There is an isometric isomorphism $F^p_\lambda(G,X)\cong F^p_\lambda(H,Y)$;
\item
There exists an isomorphism $G\ltimes X\cong H\ltimes Y$ of topological groupoids;
\item
$G\curvearrowright X$ and $H\curvearrowright Y$ are continuously orbit equivalent.
\end{enumerate}
\end{thm}
\begin{proof}
By \autoref{prp:crProdVsTransfGrpdReduced}, there are canonical isometric identifications $F^p_\lambda(G,X)\cong F^p_\lambda(G\ltimes X)$ and $F^p_\lambda(H,Y)\cong F^p_\lambda(H\ltimes Y)$.
Since the groupoids $G\ltimes X$ and $H\ltimes Y$ are topologically principal, \autoref{thm:GroupoidRigidity} implies that~(1) and~(2) are equivalent.

The equivalence between~(2) and~(3) has been noted several times in the literature;
see, for example, Theorem~1.2 of~\cite{Li18CtsOrbitEquiv}.
\end{proof}

%==========================================================================================
%==========================================================================================
\section{Tensor products of \texorpdfstring{$L^p$}{Lp}-operator algebras}
\label{sec:tensProd}

%==========================================================================================
In this section, we discuss the maximal and spatial tensor products of $L^p$-operator algebras.
Spatial tensor products have been briefly discussed 
in Remark~1.14 and Example~1.15 of~\cite{Phi13arX:LpCrProd}, and we expand on 
it here. It is not clear whether the spatial tensor product norm is the minimal $L^p$-operator
algebra tensor norm (as is the case for $C^*$-algebras), and in fact we suspect that this may be 
false in general. Maximal tensor 
products are defined in analogy with the case of
C*-algebras (see \autoref{df:TensProdLp}), and 
their norm is the largest of all $L^p$-operator
algebra tensor norms.

Given actions $G\curvearrowright X$ and $H\curvearrowright Y$, we relate the spatial (maximal) tensor product of the reduced (full) $L^p$-operator crossed products to the reduced (full) $L^p$-operator crossed product of the product action $(G\times H)\curvearrowright (X\times Y)$;
see \autoref{prp:tensCrProd}.
If both actions are amenable, then $F^p(G,X)=F^p_\lambda(G,X)$ and $F^p(H,Y)=F^p_\lambda(H,Y)$ and the reduced and maximal tensor products of $F^p_\lambda(G,X)$ and $F^p_\lambda(H,Y)$ agree;
see \autoref{prp:TensMinMax}.

%==========================================================================================
\begin{pgr}
Let $p\in[1,\infty)$, and let $A$ be a Banach algebra.
An \emph{$L^p$-representation} of $A$ is a measure space $\mu$ together with a contractive homomorphism $\pi\colon A\to\Bdd(L^p(\mu))$. (The measure $\mu$ can, without loss of 
generality, always be assumed to be localizable; see \autoref{prp:RealizeLpByLocalizable}.)
The representation is \emph{nondegenerate} if the closed linear span of $\{\pi(a)\xi\colon a\in A,\xi\in L^p(\mu)\}$ is $L^p(\mu)$.
We use $\Rep_p(A)$ to denote the class of all nondegenerate $L^p$-representations of $A$. Define a seminorm on $A$ by
setting 
\[
\|a\|_{L^p} := \sup \big\{ \|\pi(a)\| : \pi\in\Rep_p(A) \big\},
\]
for $a\in A$.
(Note that the supremum makes sense, even though $\Rep_p(A)$ is not a set.)
The \emph{enveloping $L^p$-operator algebra} of $A$, denoted by $F^p(A)$, is the Hausdorff completion of $A$ with respect to $\|\cdot\|_{L^p}$.

The Banach algebra $F^p(A)$ is an $L^p$-operator algebra with the universal property that every nondegenerate $L^p$-representation of $A$ factors through the natural map $A\to F^p(A)$.
In particular, we have a natural bijection $\Rep_p(F^p(A))\cong\Rep_p(A)$.
%If $A$ has a contractive left approximate unit (for example, if $A$ is unital), then $F^p(A)$ is also universal with respect to possibly degenerate $L^p$-representations;
%see Theorem~6.2 in \cite{GarThi20ExtendingRepr}.

Let $A$ and $B$ be Banach algebras.
We use $A\odot B$ to denote the algebraic tensor product of $A$ and $B$, and $A\tensProj B$ to 
denote their projective tensor product. 
%, which is defined as the completion of $A\odot B$ with respect to 
%\[
%\|t\|_{\hat{}} = \inf \big\{ \sum_{k=1}^n \|a_k\|\|b_k\| : t = \sum_{k=1}^n a_k\otimes b_k \big\},
%\]
%for $t\in A\odot B$.
The multiplication on $A\odot B$, given on simple tensors by $(a_1\otimes b_1)(a_2\otimes b_2):=a_1a_2\otimes b_1b_2$, extends uniquely to a multiplication on $A\tensProj B$ giving it the structure of a Banach algebra.
Given $L^p$-representations $\pi_A\colon A\to\Bdd(L^p(\mu_1))$ and $\pi_B\colon B\to\Bdd(L^p(\mu_2))$, we obtain a natural homomorphism $\pi_A\otimes\pi_B\colon A\odot B\to\Bdd(L^p(\mu_1\times\mu_2))$ satisfying $\|(\pi_A\otimes\pi_B)(a\otimes b)\|\leq\|a\|\|b\|$ for $a\in A$ and $b\in B$. In particular, $\pi_A\otimes\pi_B$ extends to an $L^p$-representation of $A\tensProj B$.
\end{pgr}

%==========================================================================================
\begin{dfn}\label{df:TensProdLp}
Let $p\in[1,\infty)$, and let $A$ and $B$ be $L^p$-operator algebras.
The \emph{spatial norm} and the \emph{maximal norm} on $A\odot B$ are 
respectively given by
\begin{align*}
\| t \|_\spat &:= \sup \big\{ \| (\pi_A\otimes\pi_B)(t) \| : \pi_A\in\Rep_p(A), \pi_B\in\Rep_p(B) \big\}, \\
\| t \|_{\text{max}} &:= \sup \big\{ \| \pi(t) \| : \pi\in\Rep_p(A\widehat{\otimes} B) \big\},
\end{align*}
for $t\in A\odot B$.
The corresponding completions of $A\odot B$ are respectively 
called the \emph{spatial $L^p$-operator algebra tensor product} (or just spatial tensor product), denoted by $A\otimes^p_\spat B$, and the \emph{maximal $L^p$-operator algebra tensor product} (or just maximal tensor product), denoted by $A\otimes^p_{\text{max}} B$.
By construction, there is a canonical contractive homomorphism
$\tau_{A,B}^p\colon A\otimes^p_{\text{max}} B\to A\otimes^p_{\text{sp}} B$
with dense range.
\end{dfn}

%==========================================================================================
%\begin{rmk}
%Given $L^p$-operator algebras $A$ and $B$, we have a natural identification $A\otimes^p_{\text{max}} B = F^p(A\tensProj B)$.
%\end{rmk}

%==========================================================================================
\begin{lma}
\label{prp:tensMaxCX}
Let $p\in[1,\infty)$, and let $X$ and $Y$ be compact, Hausdorff spaces.
Then the natural map $C(X)\odot C(Y)\to C(X\times Y)$ induces isometric isomorphisms 
\[C(X)\tensMax^p C(Y)\cong C(X)\otimes^p_{\text{sp}} C(Y)\cong C(X\times Y).\]
\end{lma}
\begin{proof}
We only need to show that the norm in $C(X)\tensMax^p C(Y)$ is dominated by the norm in $C(X\times Y)$.
Let $\mu$ be a localizable measure, and let $\pi\colon C(X)\tensProj C(Y)\to\Bdd(L^p(\mu))$ be a unital representation.
Assume that $\pi$ factors through a unital, contractive homomorphism $\rho\colon C(X)\tensMax^p C(Y)\to A$ for some \ca{} $A$.
Then $\rho(C(X)\otimes 1)$ and $\rho(1\otimes C(Y))$ are commuting, commutative, unital sub-\ca{s} of $A$, which then implies that $\rho$ factors through $C(X\times Y)$.

Thus, we need to show that $\pi$ factors through a \ca.
This is clear for $p=2$.
For $p\neq 2$, if $f\in C(X)$ is hermitian, then so is $\pi(f\otimes 1)$ by \autoref{prp:HermUnitalMaps}, and hence $\pi(f\otimes 1)$ belongs to $L^\infty(\mu)$ by \autoref{eg:HermBLp}.
Hence, $\pi(C(X)\otimes 1)\subseteq L^\infty(\mu)$. Analogously, $\pi(1\otimes C(Y))\subseteq L^\infty(\mu)$.
We deduce that $\pi$ factors through $L^\infty(\mu)$.
\end{proof}

%==========================================================================================

\begin{prp}
\label{prp:tensCrProd}
Let $p\in [1,\infty)$, and let $\mathcal{G}$
and $\mathcal{H}$ be Hausdorff, \'etale groupoids with compact unit spaces.
Then the natural map $C_c(\mathcal{G})\odot C_c(\mathcal{H})\to C_c(\mathcal{G}\times\mathcal{H})$
induces a unital, contractive homomorphism
\[
\varphi_{\lambda}^p\colon F^p_\lambda(\mathcal{G}) \otimes^p_\spat F^p_\lambda(\mathcal{H}) \to F^p_\lambda(\mathcal{G}\times \mathcal{H}),
\]
and an isometric isomorphism
\[
\varphi^p\colon F^p(\mathcal{G}) \tensMax^p F^p(\mathcal{H}) \to F^p(\mathcal{G}\times \mathcal{H}).
\]
\end{prp}
\begin{proof}
%By \autoref{prp:crProdVsTransfGrpd}, there are canonical isometric isomorphisms 
%\[
%F^p_\lambda(G,X)\cong F^p_\lambda(G\ltimes X), \quad
%F^p_\lambda(H,Y)\cong F^p_\lambda(H\ltimes Y)
%\]
%and
%\[
%F^p_\lambda(G\times H, X\times Y) \cong F^p_\lambda((G\ltimes X)\times(H\ltimes Y)).
%\]
Let $\alpha\colon C_c(\mathcal{G})\odot C_c(\mathcal{H})\to C_c(\mathcal{G}\times\mathcal{H})$ denote the natural map. Set $X=\mathcal{G}^{(0)}$ and 
$Y=\mathcal{H}^{(0)}$.
Let $x\in X$ and $y\in Y$, and let
\[
\lambda_x\colon C_c(\mathcal{G})\to\Bdd(\ell^p(\mathcal{G}x)), \andSep
\lambda_y\colon C_c(\mathcal{H})\to\Bdd(\ell^p(\mathcal{H}y))
\]
be the associated left regular representations as in \autoref{prp:regReprGroupoid}.
We have
\[
\big(\mathcal{G}\times \mathcal{H}\big)(x,y)
=(\mathcal{G}x)\times(\mathcal{H}y),
\] 
as subsets of $\mathcal{G}\times \mathcal{H}$.
Given $e\in C_c(\mathcal{G}), f\in C_c(\mathcal{H})$, $\xi\in C_c(\mathcal{G}x)$, and $\eta\in C_c(\mathcal{H}y)$, it is straightforward to check that $\alpha$ %the map $C_c(\calG x)\odot C_c(\calH y)\to C_c(\calG x \times \calH y)$ 
sends $(e\ast\xi)\otimes(f\ast\eta)$ to $(e\otimes f)\ast(\xi\otimes\eta)$.
Hence, after identifying $\ell^p(\mathcal{G}x)\otimes^p\ell^p(\mathcal{H}y)$ with $\ell^p\big((\mathcal{G}x)\times (\mathcal{H}y)\big)$, we have
\[
(\lambda_x\otimes\lambda_y)(e\otimes f) = \lambda_{(x,y)}(e\otimes f) \in \Bdd\Big( \ell^p\big((\mathcal{G}x)\times (\mathcal{H}y)\big) \Big).
\]
Note that $\lambda_x$ extends to a unital contractive representation $\lambda_x\colon F^p_\lambda(\mathcal{G})\to\Bdd(\ell^p(\mathcal{G}x))$, and similarly for $\lambda_y$.
It follows that
\begin{align*}
&\| t \|_{F^p_\lambda(\mathcal{G}) \otimes^p_\spat F^p_\lambda(\mathcal{H})} \\
&\quad\quad 
= \sup \big\{ \| (\pi_1\otimes\pi_2)(t) \| : \pi_1\in\Rep_p(F^p_\lambda(\mathcal{G})), \pi_2\in\Rep_p(F^p_\lambda(\mathcal{H})) \big\} \\
&\quad\quad 
\geq \sup \big\{ \| (\lambda_x\otimes\lambda_y)(t) \| : x\in X, y\in Y \big\} \\
&\quad\quad = \sup \big\{ \| \lambda_{(x,y)}(t) \| : (x,y)\in X\times Y \big\} 
= \|t\|_{F^p_\lambda(\mathcal{G}\times \mathcal{H})},
\end{align*}
for every $t\in C_c(\mathcal{G})\odot C_c(\mathcal{H})$, which proves the first statement.

The statement about the maximal tensor product of full groupoid $L^p$-operator
algebras follows easily using the universal properties of the objects involved:
%Unital representations of $F^p(G,X)$ naturally correspond to covariant representations of $(G,X)$, and similarly for $F^p(H,Y)$.
Unital representations of $F^p(\mathcal{G})\tensProj F^p(\mathcal{H})$ correspond to pairs consisting of commuting groupoid representations of $\mathcal{G}$ and 
$\mathcal{H}$, which are easily seen to correspond to groupoid representations of $\mathcal{G}\times \mathcal{H}$. We omit the straightforward details.
\end{proof}

%==========================================================================================
In the context of the proposition above, it 
is not clear if $\varphi_{\lambda}^p$ is isometric,
except for the situations covered by \autoref{prp:TensMinMax}.
In particular, given nonamenable groups $G$ and $H$, it is not clear if $F^p_\lambda(G)\otimes^p_\spat F^p_\lambda(H)$ is isometrically isomorphic to $F^p_\lambda(G\times H)$.

%==========================================================================================

We record here the following useful fact, which is the crossed product analog of a
similar result for \'etale groupoids, namely Theorem~6.19 in~\cite{GarLup17ReprGrpdLp}.
(Observe that the lemma below does not directly follow from Theorem~6.19 in~\cite{GarLup17ReprGrpdLp}, since we do not know in general whether 
$F^p(G,X)$ is isometrically isomorphic to $F^p(G\ltimes X)$.)

\begin{lma}
\label{prp:amenableCrProd}
Let $p\in [1,\infty)$, and let $G\curvearrowright X$ be an amenable action of a discrete group $G$ on a compact, Hausdorff space $X$.
Then the canonical contractive homomorphism 
$\kappa_{(G,X)}^p\colon F^p(G,X) \to F^p_\lambda(G,X)$ is an isometric isomorphism.
\end{lma}
\begin{proof}
This is proved identically to the implication (1) $\Rightarrow$ (2) 
of Theorem~5.3 in~\cite{Ana02AmenExactDynSysCa}. We omit the details.
\end{proof}

%==========================================================================================
We conclude this section with the following result on tensor products of
amenable groupoids.

\begin{thm}
\label{prp:TensMinMax}
Let $p\in [1,\infty)$, and let $\mathcal{G}$ and $\mathcal{H}$ be 
amenable, \'etale, Hausdorff groupoids with compact unit spaces.
There are natural isometric isomorphisms
\[
F^p_\lambda(\mathcal{G}) \tensMax^p F^p_\lambda(\mathcal{H}) 
\cong F^p_\lambda(\mathcal{G}) \otimes^p_\spat F^p_\lambda(\mathcal{H})
\cong F^p_\lambda(\mathcal{G}\times \mathcal{H}).
\]
\end{thm}
\begin{proof}
We write $\kappa^p_{\mathcal{G}}$, $\kappa^p_{\mathcal{H}}$ and 
$\kappa^p_{\mathcal{G}\times\mathcal{H}}$ for the canonical unital, 
contractive homomorphisms with dense range from the full to the reduced groupoid 
$L^p$-operator algebras of the groupoids in question.
We write $\tau^p_{\mathcal{G},\mathcal{H}}\colon F^p_\lambda(\mathcal{G}) \tensMax^p F^p_\lambda(\mathcal{H})\to
F^p_\lambda(\mathcal{G}) \otimes^p_\spat F^p_\lambda(\mathcal{H})$ for the canonical map.
Using the maps $\varphi_\lambda^p$ and $\varphi^p$ from \autoref{prp:tensCrProd},
we obtain the following commutative diagram:
\[
\xymatrix@C-10pt@R-5pt{
&& F^p(\mathcal{G}) \tensMax^p F^p(\mathcal{H}) \ar[ddll]_-{\kappa_{\mathcal{G}}^p\otimes \kappa_{\mathcal{H}}^p}
\ar[rr]^-{\varphi^p}
&& F^p(\mathcal{G}\times\mathcal{H}) \ar[dd]^{\kappa^p_{\mathcal{G}\times \mathcal{H}}} \\
\\
F^p_\lambda(\mathcal{G}) \tensMax^p F^p_\lambda(\mathcal{H}) \ar[rr]_{\tau^p_{\mathcal{G},\mathcal{H}}}
&& F^p_\lambda(\mathcal{G}) \otimes^p_\spat F^p_\lambda(\mathcal{H})
\ar[rr]_-{\varphi^p_\lambda}
&& F^p_\lambda(\mathcal{G}\times\mathcal{H}).\\
}
\]
Observe that $\kappa_{\mathcal{G}}^p\otimes \kappa_{\mathcal{H}}^p$, $\varphi^p$ and 
$\kappa^p_{\mathcal{G}\times \mathcal{H}}$ are isometric isomorphisms (see \autoref{prp:amenableCrProd} and \autoref{prp:tensCrProd}).
Thus, the identity
\[\kappa^p_{\mathcal{G}\times \mathcal{H}}\circ \varphi^p\circ (\kappa_{\mathcal{G}}^p\otimes \kappa_{\mathcal{H}}^p)^{-1}=\varphi^p_\lambda\circ \tau^p_{\mathcal{G},\mathcal{H}}\]
implies that $\varphi^p_\lambda\circ \tau^p_{\mathcal{G},\mathcal{H}}$ is an isometric isomorphism. Since $\tau^p_{\mathcal{G},\mathcal{H}}$ and $\varphi^p_\lambda$ are contractive,
we deduce that they must be isometric, and hence also isomorphisms (since
their ranges are dense). This finishes the proof.
\end{proof}

%==========================================================================================
%==========================================================================================
\section{Tensor products of \texorpdfstring{$L^p$}{Lp}-Cuntz algebras}
\label{sec:tensCuntz}

%==========================================================================================
Tensor products of Cuntz algebras have played a pivotal role in the study of the structure and classification of simple, purely infinite, nuclear \ca{s} (also known as \emph{Kirchberg algebras}).
A particularly remarkable result in this direction, which was instrumental in the classification results of Kirchberg and Phillips, is Elliott's theorem that $\Ot\cong \Ot\otimes\Ot$;
see \cite{Ror94ShortElliottsThem} for a self-contained account.

In \cite{Phi12arX:LpAnalogsCtz}, Phillips introduced $L^p$-analogs $\Onp$ of the Cuntz algebras, and proved that these $L^p$-operator algebras share many remarkable properties with their $C^*$-versions. 
It is then natural to explore the extent to which the $K$-theoretic classification theory for Kirchberg algebras can be extended to the $L^p$-setting;
in particular, it becomes indispensable to know whether 
$\Otp$ is isometrically isomorphic to its tensor
square (with respect to either $\tensMax^p$ or $\tensSp^p$).
In this section, we show that this is not the case, and deduce that purely infinite, simple, amenable $L^p$-operator algebras are not classified by $K$-theory.
This answers several questions of Phillips.

We begin by recasting Phillips' construction of $\Onp$.

%==========================================================================================
\begin{pgr}
Let $A$ be a unital Banach algebra.
Recall that $A_{\mathrm{h}}$ denotes the set of hermitian elements in $A$ (see 
the beginning of Section~2 for the definition). 
Let $a\in A$.
An element $b\in A$ is called a \emph{Moore-Penrose inverse} of~$a$ if $a=aba$ and $b=bab$, and if $ab,ba\in A_{\mathrm{h}}$.
It is well-known that $a$ has at most one Moore-Penrose inverse, which allows us to denote it by $a^\dagger$ (if it exists).

Following Mbekhta, \cite{Mbe04PartIsom}, we say that $a\in A$ is a \emph{MP-partial isometry} if $a$ is contractive and has a contractive Moore-Penrose inverse.
If $a$ is a MP-partial isometry, then so is $a^\dagger$, and we have $(a^\dagger)^\dagger=a$. (If $A$ is a \ca\ and $a\in A$, then $a$ is a 
MP-partial isometry if and only if $a^*a$ is a projection, in which case 
$a^\dagger=a^*$.)

%If $A$ is a \ca{}, then $a\in A$ is a partial isometry in the usual sense (that is, $a^*a$ is a projection) if and only if it a MP-partial isometry (and then $a^\dagger=a^*$).

Let $p\in[1,\infty)$, let $\mu$ be a ($\sigma$-finite) measure space.
Then $a\in \Bdd(L^p(\mu))$ is a MP-partial isometry if and only if $a$ is a \emph{spatial partial isometry} in the sense of Definition~6.4 in \cite{Phi12arX:LpAnalogsCtz}.
\end{pgr}

%==========================================================================================
Let us recall the necessary notions from Definition~7.4(2) in \cite{Phi12arX:LpAnalogsCtz}.

%==========================================================================================
\begin{dfn}
\label{dfn:Ln}
Let $n\in\NN$ with $n\geq 1$.
The \emph{Leavitt algebra} $L_n$ is the universal unital complex algebra generated by elements $s_1,\ldots,s_n,t_1,\ldots,t_n$ satisfying
\[
t_js_k=\delta_{j,k}, \andSep 
\sum_{j=1}^n s_jt_j=1,
\]
for $j,k=1,\ldots,n$.
Let $p\in [1,\infty)$, and let $E$ be an $L^p$-space.
A \emph{spatial representation} of $L_n$ on $E$ is a unital homomorphism $\rho\colon L_n\to \Bdd(E)$ such that $\rho(s_j)$ is a MP-partial isometry with $\rho(s_j)^\dagger=\rho(t_j)$, for all $j=1,\ldots,n$.
\end{dfn}

%==========================================================================================
Examples of spatial representations are easy to construct using shift operators on $\ell^p(\NN)$.
By Theorem~8.7 in~\cite{Phi12arX:LpAnalogsCtz}, if $\rho_1$ and $\rho_2$ are spatial representations of $L_n$ on $L^p$-spaces, then $\|\rho_1(x)\|=\|\rho_2(x)\|$ for all $x\in L_n$.
The \emph{$L^p$-Cuntz algebra} $\Onp$ is then defined as $\Onp=\overline{\rho(L_n)}$ for any spatial representation $\rho$.
For $p=2$, one gets the usual Cuntz $C^*$-algebra $\mathcal{O}_n$ from \cite{Cun77SimpleCAlgGenIsom}.

It was observed in \cite{GarLup17ReprGrpdLp} that $\Onp$ is the groupoid algebra
associated to the groupoid of a graph.
We will need to realize $\Onp$ as the algebra associated to a transformation groupoid, and we begin by introducing some notation. 

%==========================================================================================
\begin{pgr}
A directed graph $E=(E^0,E^1,r,s)$ is a set $E^0$ of vertices and a set $E^1$ of edges together with source and range maps $s,r\colon E^1\to E^0$.
We assume that $E^0$ and $E^1$ are finite and that $E$ has no sinks, that is $s^{-1}(v)\neq\emptyset$ for all $v\in E^0$.
Set
\[
E^\infty := \big\{ (x_1,x_2,\ldots)\colon x_k\in E^1, r(x_k)=s(x_{k+1}) \text{ for all } k\geq 1 \big\}.
\]
We equip $E^\infty$ with the topology inherited from the product topology on $(E^1)^\NN$, which turns it into a totally disconnected, compact, Hausdorff space.
Define the shift map $\sigma_E\colon E^\infty\to E^\infty$ by
$\sigma_E (x_1,x_2,\ldots) = (x_2,x_3,\ldots)$ for all $(x_1,x_2,\ldots)\in E^\I$.

The \emph{graph groupoid} $\calG_E$ associated to $E$ is defined as
\[
\calG_E := \left\{ (x,k,y) \in E^\infty\times\ZZ\times E^\infty\colon \begin{aligned}
&\text{there are } m,n\geq 0 \text{ satisfying }\\
& k=m-n \mbox{ and } \sigma_E^m(x)=\sigma_E^n(y) \end{aligned}
\right\},
\]
together with range and source maps given by
$r (x,k,y)  = (x,0,x)$ and $s(x,k,y)  = (y,0,y)$,
and composition and inversion given by
\[
(x,k,y)(y,l,z) = (x,k+l,z), \andSep
(x,k,y)^{-1} = (y,-k,x).
\]
We equip $\calG_E$ with the topology inherited from the product topology on $E^\infty\times\ZZ\times E^\infty$.
Then $\calG_E$ is a locally compact, Hausdorff, \'etale groupoid.
Its unit space is $\calG_E^{(0)}=\{(x,0,x):x\in E^\infty\}$, which we identify with $E^\infty$;
see also Example~2.4.7 in \cite{Sim17arX:EtaleGpds} (but note the subtle difference in the definition of $E^\infty$ with the direction of arrows in a path reversed).
\end{pgr}

\begin{dfn}
Given $n\geq 2$, let $E_n$ be the graph with one vertex and $n$ edges (loops at the vertex). We define the \emph{Cuntz groupoid} $O_n$ to be the groupoid associated to $E_n$ as in the paragraph above.
\end{dfn}

%==========================================================================================
Next, we realize $O_2$ as a transformation groupoid. The identification is probably known
to the experts, but we were not able to find a
suitable reference. 
%The result is implicit in the work of Spielberg \cite{Spi91FreeProd} but we include a proof for the convenience of the reader.

%==========================================================================================
\begin{prp}
\label{prp:CtzGrpdAsCrProd}
There exists an amenable, topologically free action of $\ZZ_2\ast\ZZ_3$ on the Cantor space $X$ such that $O_2\cong(\ZZ_2\ast\ZZ_3)\ltimes X$ as topological groupoids.
\end{prp}
\begin{proof}
Consider the following directed graphs:
\begin{eqnarray*}
\xymatrix{
F: & \bullet \ar[r]^{e} & \bullet \ar@(ul,ur)^{f_1}\ar@(dl,dr)_{f_2} 
& & &
E_2: & \bullet \ar@(ul,ur)^{f_1}\ar@(dl,dr)_{f_2}
}
\end{eqnarray*}

\textbf{Claim~1}: \emph{We have $\calG_F \cong \calG_{E_2}$.}
Define $\varphi\colon F^\infty\to E_2^\infty$ and $\varepsilon\colon F^\infty\to\{1,2\}$ by
\[
\varphi(x_1,x_2,x_3,\ldots)
= \begin{cases}
(f_1,x_2,x_3,\ldots) & \text{if } x_1=e, \\
(f_2,x_1,x_2,\ldots) & \text{if } x_1\neq e.
\end{cases}
\quad
\varepsilon(x_1,x_2,\ldots)
= \begin{cases}
1 & \text{if } x_1=e, \\
2 & \text{if } x_1\neq e.
\end{cases}
\]
Then $\varphi$ is a homeomorphism.
Further, note that $\sigma_{E_2}^{\varepsilon(x)}(\varphi(x))=\sigma_F(x)$ for every $x\in F^\infty$.
Finally, define $\Phi\colon\calG_F\to\calG_{E_2}$ by
\[
\Phi(x,k,y) := (\varphi(x),k-\varepsilon(x)+\varepsilon(y),\varphi(y)).
\]
The map is well-defined since if $m,n\geq 0$ satisfy $k=m-n$ and $\sigma_F^m(x)=\sigma_F^n(y)$, then $k-\varepsilon(x)+\varepsilon(y) = [m+\varepsilon(y)]-[n+\varepsilon(x)]$ and
\[
\sigma_{E_2}^{m+\varepsilon(x)}(\varphi(x))
= \sigma_{F}^{m+1}(x)
= \sigma_F^{n+1}(y)
= \sigma_{E_2}^{n+\varepsilon(y)}(\varphi(y)).
\]

It is straightforward to check that $\Phi$ is surjective and compatible with the range and source maps and with composition. This proves the claim.
\vspace{.2cm}

\textbf{Claim~2:} 
\emph{there is an action $\ZZ_2\ast\ZZ_3\curvearrowright F^\infty$ satisfying $\calG_F \cong (\ZZ_2\ast\ZZ_3)\ltimes F^\infty$.}
Let $a\in\ZZ_2$ and $b\in\ZZ_3$ denote generators.
Set:
\begin{align*}
a\cdot(x_1,x_2,x_3,\ldots)
&= \begin{cases}
(e,x_1,x_2,x_3,\ldots) & \text{if } x_1\neq e, \\
(x_2,x_3,\ldots) & \text{if } x_1=e.
\end{cases}
\\
b\cdot (x_1,x_2,x_3,\ldots)
&= \begin{cases}
(e,f_1,x_1,x_2,x_3,\ldots) & \text{if } x_1\neq e, \\
(e,f_2,x_3,\ldots) & \text{if } x_1=e, x_2=f_1, \\
(x_3,\ldots) & \text{if } x_1=e, x_2=f_2, \\
\end{cases}
\end{align*}
It is straightforward to verify that $a$ and $b$ define homeomorphisms on $F^\infty$ of order two and three, respectively.
We thus obtain an action $\ZZ_2\ast\ZZ_3\curvearrowright F^\infty$.

We identify $E_2^\infty$ with a subset of $F^\infty$.
Further, we let $eE_2^\infty$, $ef_1E_2^\infty$ and $ef_2E_2^\infty$ denote the clopen subsets of $F^\infty$ consisting of sequences starting with $e$, with $ef_1$, and with $ef_2$, respectively.
Define $\varepsilon_a,\varepsilon_b\colon F^\infty\to\ZZ$ by
\[
\varepsilon_a(x)
= \begin{cases}
+1 & \text{if } x\in E_2^\infty, \\
-1 & \text{if } x\in eE_2^\infty,
\end{cases}
\quad
\varepsilon_b(x)
= \begin{cases}
+2 & \text{if } x\in E_2^\infty, \\
0 & \text{if } x\in ef_1E_2^\infty, \\
-2 & \text{if } x\in ef_2E_2^\infty.
\end{cases}
\]
For $g=g_1g_2\ldots g_n\in\ZZ_2\ast\ZZ_3$,
with $g_j\in \{a,b\}$, we let
$\varepsilon_g\colon F^\infty\to\ZZ$ be given by
\[
\varepsilon_{g_1g_2\ldots g_n}(x)
:=\varepsilon_{g_1}(g_2\cdots g_n\cdot x)+\varepsilon_{g_2}(g_3\cdots g_n\cdot x)+\ldots+ \varepsilon_{g_n}(x).
\]
One checks that $\varepsilon_g$ is well-defined,
and that the conditions $g\cdot x=x$ and $\varepsilon_g(x)=0$ imply $g=1$. 
Given $g\in\ZZ_2\ast\ZZ_3$, we define $\Psi_g\colon \calG_F\to\calG_F$ by
\[
\Psi_g(x,k,y)=(gx,k+\varepsilon_g(x),y).
\]
Then $\Psi_g$ is a well-defined bijection and $\Psi_{gh}=\Psi_g\circ\Psi_h$ for $g,h\in\ZZ_2\ast\ZZ_3$. 
Define a map $\Omega\colon(\ZZ_2\ast\ZZ_3)\ltimes F^\infty\to\calG_F$ by
\[
\Omega(g,x)=\Psi_g(x,0,x) = (gx, \varepsilon_g(x),x).
\]
We want to show that $\Omega$ is an isomorphism of topological groupoids.
To check that~$\Omega$ preserves composition of arrows, let $g,h\in\ZZ_2\ast\ZZ_3$ and $x\in F^\infty$.
Then
\begin{align*}
\Omega(g,hx)\Omega(h,x)
&= (ghx, \varepsilon_{g}(hx),hx) (hx, \varepsilon_h(x),x) \\
&= (ghx, \varepsilon_{g}(hx)+\varepsilon_h(x),x) \\
&= (ghx, \varepsilon_{gh}(x),x)
= \Omega(gh,x),
\end{align*}
as desired. Moreover, $\Omega$ is injective since
$\Omega(g,x)=\Omega(h,y)$ implies 
$x=y$, $(gh^{-1})\cdot x=x$, and $\varepsilon_{gh^{-1}}(x)=0$ (the last two together imply $g=h$). Surjectivity is proved using similar
arguments, and is left to the reader. This proves the claim.
\vspace{.2cm}

It follows from Claims~1 and~2 that 
$O_2=\calG_{E_2}\cong(\ZZ_2\ast\ZZ_3)\ltimes X$. 
Finally, the facts that the action $\ZZ_2\ast\ZZ_3\curvearrowright X$ is 
amenable and topologically free follow, respectively,
from the facts that $O_2$ is amenable and effective (see Theorem~3.8 and Example~4.4 in~\cite{Ana02AmenExactDynSysCa} and Theorem~4.3.6 in~\cite{Sim17arX:EtaleGpds}).
\end{proof}

\begin{rmk}\label{rmk:ZkZn+1}
We mention without proof that the construction above
can be generalized to show that for every $k\in\NN$ and every $n\geq 2$,
the groupoid $M_k(O_n)$ can be realized as 
the transformation groupoid of an amenable, topologically free
action of $\mathbb{Z}_k\ast \mathbb{Z}_{n+1}$ on the Cantor space. (Implicit in the proof above is the fact that $M_2(O_2)$ and $)_2$ are isomorphic as groupoids.)
\end{rmk}

%==========================================================================================
Combining \autoref{prp:CtzGrpdAsCrProd}, \autoref{prp:amenableCrProd} and \autoref{prp:TensMinMax}, we deduce that when taking tensor products
of $\Otp$ with itself, we may choose either $\tensSp$ or $\tensMax$:

%==========================================================================================
\begin{cor}
\label{prp:CtzAlgsCrossedProd}
There exists an amenable, topologically free action of $\ZZ_2\ast\ZZ_3$ on the Cantor space $X$ such that, for every $p\in [1,\I)$, there are
isometric isomorphisms
\[
\Otp
\cong F^p(\ZZ_2\ast\ZZ_3, X)
= F^p_\lambda(\ZZ_2\ast\ZZ_3, X).
\]
Furthermore, for $n\geq 1$ we have
\[
\underbrace{\Otp\tensMax^p\cdots\tensMax^p\Otp}_{n}
\cong \underbrace{\Otp \otimes^p_\spat\cdots\otimes^p_\spat \Otp}_{n}
\cong F^p_\lambda((\ZZ_2\ast\ZZ_3)^n, X^n).
\]
\end{cor}

%==========================================================================================
Since the choice of tensor product is irrelevant when taking tensor products
of $\Otp$ with itself, we will from now on 
just write 
$\Otp\otimes^p\cdots\otimes^p\Otp$.

We make a small digression to establish some facts from geometric group theory that will be needed.
In \cite[Theorem~3.2]{MedSauTom17CantorSysQuasiIsom} it is shown that finitely generated groups are bi-Lipschitz equivalent (see Definition~2.6 in~\cite{MedSauTom17CantorSysQuasiIsom}) if and only if they admit free actions on the Cantor set that are continuously orbit equivalent.
An inspection of their proof shows that it suffices 
to assume \emph{topological freeness} of the actions
in order to conclude that the groups are
bi-Lipschitz equivalent\footnote{When defining the map between groups, just pick a point with trivial stabilizer. 
This is possible whenever the space is Baire; 
see, for example, the proof of Proposition~4.6 in \cite{GarGefKraNar23ClassCrProdNonamenGp}.}, and that the spaces are compact and Hausdorff. 
We may thus restate their result as follows:

%==========================================================================================
\begin{thm}[Medynets-Sauer-Thom]
\label{prp:MST}
Let $G$ and $H$ be finitely generated groups.
Then $G$ and $H$ are bi-Lipschitz equivalent
if and only if there are continuously orbit equivalent, topologically free actions 
of $G$ and $H$ on compact Hausdorff spaces.
\end{thm}

%==========================================================================================
We use $\asdim(G)$ to denote the asymptotic dimension of a group $G$;
we refer to \cite{BelDra08Asdim} for the definition and the basic properties of this dimension theory.
It is well-known that if two finitely generated groups $G$ and $H$ are bi-Lipschitz equivalent, then they are coarsely equivalent and therefore $\asdim(G)=\asdim(H)$.
Thus, combining \autoref{thm:RigidityDynSysts} and \autoref{prp:MST}, we obtain:

%==========================================================================================
\begin{cor}
\label{prp:IsoCrProdBiLip}
Let $p\in [1,\infty)\setminus\{2\}$, let $G$ and $H$ be finitely generated groups, 
let $X$ and $Y$ be compact Hausdorff spaces, and let $G\curvearrowright X$ and $H\curvearrowright Y$ be topologically free actions such that $F^p_\lambda(G,X)$ and $F^p_\lambda(H,Y)$ are isometrically isomorphic.
Then $G$ and $H$ are bi-Lipschitz equivalent.
In particular, $\asdim(G)=\asdim(H)$.
\end{cor}

%==========================================================================================
The following result is well-known to experts. 

%==========================================================================================
\begin{lma}
\label{prp:asdimZ2Z3}
Let $n\in\NN$.
Then $\asdim((\ZZ_2\ast \ZZ_{3})^n)=n$.
\end{lma}
\begin{proof}
Let $G_1,\ldots,G_n$ be finitely generated groups satisfying $\asdim(G_k)=1$ for $k=1,\ldots,n$.
%Set $G:=G_1\times\ldots\times G_n$.
We claim that $\asdim(\prod_{k=1}^n G_k)=n$.

By Theorem~32 in \cite{BelDra08Asdim}, we have $\asdim(H_1\times H_2)\leq\asdim(H_1)+\asdim(H_2)$ for all (finitely generated) groups $H_1$ and $H_2$.
We deduce $\asdim(\prod_{k=1}^n G_k)\leq n$.
The converse inequality follows from Theorem~1 in~\cite{BanBan22AsymDimProd}.
Now the statement follows from the fact $\asdim(\ZZ_2\ast \ZZ_{3})=1$; see, for instance, Section~17 of~\cite{BelDra08Asdim}.
\end{proof}

%==========================================================================================
We have arrived at the main result of this section.

\begin{thm}
\label{prp:O2O2Noniso}
Let $p\in [1,\infty)\setminus\{2\}$, and let $m,n\in\NN$.
Then there is an isometric isomorphism
\[
\underbrace{\Otp\otimes^p\cdots\otimes^p\Otp}_\text{m} \cong \underbrace{\Otp\otimes^p\cdots\otimes^p\Otp}_\text{n} 
\]
if and only if $m=n$.
In particular, $\Otp$ is not isometrically isomorphic to $\Otp\otimes^p\Otp$.
\end{thm}
\begin{proof}
We need to show the forward implication.
Assume that there is an isometric isomorphism as in the statement.
Let $\ZZ_2\ast\ZZ_3\curvearrowright X$ be the topologically free action on the Cantor space $X$ as in \autoref{prp:CtzAlgsCrossedProd}.
It follows that the $m$-fold and the $n$-fold power of this action have isometrically isomorphic reduced crossed products.
Applying \autoref{prp:asdimZ2Z3} and \autoref{prp:IsoCrProdBiLip}, we obtain
\[
m
= \asdim((\ZZ_2\ast \ZZ_{3})^m)
= \asdim((\ZZ_2\ast \ZZ_{3})^n)
= n. \qedhere
\]
\end{proof}

%==========================================================================================
For $p=2$, it is known that $\Ot\otimes\cdots\otimes \Ot\cong \Ot$ as \ca{s}, by iterating the theorem of Elliott mentioned at the beginning of this section.
However, the isomorphism is produced very indirectly, and there is no known explicit formula for it (for example, in terms of the canonical generators).
Further, it is a folklore result, which is implicitly contained in early work of Cuntz \cite{Cun80AutCertainSimpleCAlg}, that there is no isomorphism between  $\Ot\otimes\cdots\otimes \Ot$ and $\mathcal{O}_2$ that preserves the canonical Cartan subalgebras.
We sketch a proof based on the results in \cite{Cun80AutCertainSimpleCAlg}, and we also include \autoref{thm:O2O2Cartan} below with a proof based on geometric group theory.

We abbreviate the $n$-fold tensor product of $\Ot$ with itself by $\mathcal{O}_2^{\otimes n}$. 
By \cite[Proposition~3.1]{Cun80AutCertainSimpleCAlg}, a pure state on the canonical Cartan subalgebra $D_2$ in $\mathcal{O}_2$ either has a unique extension to a pure state on $\mathcal{O}_2$, or the family of extensions to pure states on $\mathcal{O}_2$ is homeomorphic to $\mathbb{T}$, and both cases occur.
It follows that there are pure states on the canonical Cartan subalgebra $D_2^{\otimes n}$ in $\mathcal{O}_2^{\otimes n}$ whose set of extensions to a pure state on $\mathcal{O}_2^{\otimes n}$ contains a subset homeomorphic to $\mathbb{T}^n$.
Since this is not the case for any pure state of the Cartan subalgebra~$D_2$ in~$\mathcal{O}_2$, this shows that there is no isomorphism between $\mathcal{O}_2^{\otimes n}$ and $\mathcal{O}_2$ that preserves the canonical Cartan subalgebras.

%==========================================================================================
\begin{thm}
\label{thm:O2O2Cartan}
No C*-algebraic isomorphism between
$\mathcal{O}_2^{\otimes m}$ and $\mathcal{O}_2^{\otimes n}$ preserves the canonical Cartan subalgebras if $m \neq n$.
\end{thm}
\begin{proof}
This is essentially the same proof as for 
\autoref{prp:O2O2Noniso}; there, the property corresponding to preservation of the
Cartan subalgebras (preservation of the C*-cores)
is automatic, since $p\neq 2$. 
Instead of applying
\autoref{prp:IsoCrProdBiLip},
one uses Theorem~1.2 in~\cite{Li18CtsOrbitEquiv}
to deduce that the systems
$(\ZZ_2\ast\ZZ_3)^m\curvearrowright X^m$ and 
$(\ZZ_2\ast\ZZ_3)^n\curvearrowright X^n$
provided by \autoref{prp:CtzAlgsCrossedProd}
are continuously orbit equivalent. 
The result is then obtained by applying \autoref{prp:asdimZ2Z3} and \autoref{prp:IsoCrProdBiLip}.
\end{proof}

%==========================================================================================
%??  Should we mention that $L_2\otimes L_2\ncong L_2$?  THIS SHOULD
%GO IN THE INTRO.

%==========================================================================================

For comparison, we mention that Ara and Corti\~nas have shown in 
\cite{AraCor13TensLeavitt} that $L_2\odot L_2$
is not isomorphic to $L_2$. Their methods are 
quite different from ours; in fact, the invariant
they used to distinguish $L_2\odot L_2$ and $L_2$, Hochschild homology,
cannot distinguish between $\Otp\otimes^p\Otp$ 
and $\Otp$. 
We do not know \emph{any} homotopy-invariant functor that is able to distinguish between
$\Otp\otimes^p\Otp$ and $\Otp$ when $p\neq 2$.
In particular, these algebras are not distinguishable by $K$-theory, as we show in
\autoref{prop:KtheoryTheSame} below. As a 
preparatory result, we show that 
$M_2^p\otimes^p\Otp$ is 
isometrically isomorphic to $\Otp$. Since
the proof is the same, we do it in greater 
generality. 
Recall that if $\varphi_0$ is any spatial representation of a Leavitt path algebra $L_n$ on an $L^p$-space, then $\varphi_0$
extends to an isometric representation of $\Onp$.
%(\textbf{Does the following result hold for Leavitt algebras?? Include remark??})

%==========================================================================================
\begin{prp}
\label{prop:OnpMatrices}
Let $k, r \in\NN$ with $k\geq 1$, and let $p\in [1,\infty)$.
Then $M^p_{2^r}\otimes^p \mathcal{O}_{2k}^p$ is isometrically isomorphic to $\mathcal{O}_{2k}^p$.
\end{prp}
\begin{proof}
By finite induction, it is clearly enough to prove the result for $r=1$.
For $j=1,\ldots,k$, we define
\[
x_{2j-1}=\begin{pmatrix}
s_j & s_{j+1} \\
0 & 0
\end{pmatrix}, \ \mbox{ and } \ x_{2j}=\begin{pmatrix}
0 & 0 \\
s_j & s_{j+1}
\end{pmatrix},
\]
and their reverses
\[
y_{2j-1}=\begin{pmatrix}
t_j & 0 \\
t_{j+1} & 0
\end{pmatrix}, \ \mbox{ and } \ y_{2j-1}=\begin{pmatrix}
0 & t_j  \\
0 & t_{j+1}
\end{pmatrix}.
\]
One checks that these are spatial partial isometries satisfying the relations in the definition of $\mathcal{O}_{2k}^p$.
By the universal property of $L_{2k}$, there is a unital homomorphism
$\varphi_0\colon L_{2k}\to M_2^p \otimes^p \mathcal{O}_{2k}^p$ defined by
$\varphi_0(s_j)=x_j$ and $\varphi_0(t_j)=y_j$ for all $j=1,\ldots,2k$.
One easily checks that $\varphi_0$ is spatial (in the sense of the comments after \autoref{dfn:Ln}),
and hence it extends to an isometric homomorphism 
\[
\varphi\colon \mathcal{O}^p_{2k}\to M_2^p \otimes^p \mathcal{O}_{2k}^p.
\]
Since the elements $x_1,\ldots,x_{2k}, y_1,\ldots,y_{2k}$
generate all of $M_2^p \otimes^p \mathcal{O}_{2k}^p$, we deduce that $\varphi$
has dense range and hence is an isometric isomorphism.
\end{proof}

For $p\in [1,\infty)$, we set 
\[
\overline{M}_\infty^p=\overline{\bigcup\limits_{n\in\NN} \Bdd(\ell^p(\{1,\ldots,n\}))}\subseteq \Bdd(\ell^p(\NN)).
\]
For $p>1$, it is known that $\overline{M}_\infty^p$ agrees with $\K(\ell^p(\NN))$, but this fails for $p=1$ (see Example~1.10 in~\cite{Phi13arX:LpCrProd}).
Regardless of $p$, matrix stability of $K$-theory together with continuity 
with respect to inductive limits shows that 
$K_\ast(A\otimes^p_\spat \overline{M}_\infty^p)$ is isomorphic to 
$K_\ast(A)$ for any algebra $A$. We will need the following
observation.

\begin{rmk}
If $D$ is a direct limit of algebras of the form $\Bdd(\ell^p(\{1,\ldots,n\}))$,
for $n\in\NN$, then $A\tensMax^p D$ and $A\tensSp^p D$ are canonically isometrically isomorphic for every $L^p$-operator algebra $A$.
This is essentially immediate 
for the matrix algebra $\Bdd(\ell^p(\{1,\ldots,n\}))$, and the result for 
$D$ is obtained by taking direct limits. In particular, this applies to
$\overline{M}_\infty^p$ as well as to any spatial UHF-algebra 
(\cite{Phi13arX:LpCrProd}).
\end{rmk}

%==========================================================================================
\begin{prp}
\label{prop:KtheoryTheSame}
Let $p\in [1,\infty)$, let $n\in\NN$, and let $A_n$ denote the $n$-fold tensor product of $\Otp$ with itself. 
Then $A_n$ is simple, purely infinite, and amenable, with
\[
K_0(A_n)=K_1(A_n)=\{0\}.
\]
Moreover, $A_n$ is isometrically isomorphic to $A_m$ if and only if $n=m$.
\end{prp}
\begin{proof}
%(\textbf{This proof seems incomplete.})
Since $\Otp$ is amenable, so is the $n$-fold projective tensor product $\Otp \widehat{\otimes} \dots \widehat{\otimes} \Otp$. The natural homomorphism from this Banach algebra to $A_n$ is continuous with dense range, so $A_n$ is amenable. 
%(\textbf{Why does amenability passes to tensor products??})
Simplicity and pure infiniteness follow from Theorem~7.9 in~\cite{AraGooPerSil10NonSimplePI},
since the combination of simplicity and
pure infiniteness passes from a dense subring to the containing Banach algebra.

It remains to compute the $K$-theory of $A_n$, which we do by induction on $n$.
For $n=1$, this was shown by Phillips in Theorem~7.19 of~\cite{Phi13arX:LpCrProd}.
Assume that we have proved the result for $A_n$, and let us show it for $A_{n+1}=A_n\otimes^p\Otp$.

Denote by $B$ the tensor product of the spatial 
$L^p$-UHF algebra of type $2^{\infty}$ with $\overline{M}_\infty^p$, 
identified with the tensor product $\bigotimes_{k\in\ZZ}M_2^p$ as in Section~7 of~\cite{Phi13arX:LpCrProd}, 
and let $\beta\colon \ZZ\to\Aut(B)$ denote the bilateral shift.
By Theorem~7.17 in~\cite{Phi13arX:LpCrProd}, there is an isometric
isomorphism $\Otp\otimes^p\overline{M}_\infty^p\cong F^p(\ZZ,B,\beta)$, and thus
\[
(\Otp\otimes^p \overline{M}_\infty^p)\tensMax^p {A_n}
\cong F^p(\ZZ,B,\beta)\tensMax^p A_n,
\]
so in particular these two algebras have isomorphic $K$-theory.
An argument identical to the one given in the second part of 
\autoref{prp:TensMinMax}, using universal properties, shows 
that 
%$F^p(\ZZ,B,\beta)\tensMax^p A_n\cong F^p(\ZZ,B,\beta)\tensSp^p A_n
the right-hand side can be canonically identified with
the full crossed product of the action $\beta\otimes\id_{A_n}\colon \ZZ\to \Aut(B\otimes^p {A_n})$.

We claim that $B\otimes^p {A_n}$ has trivial $K$-theory. 
Since $B=M_{2^\infty}^p\otimes^p \overline{M}_\infty^p$, it suffices to show that $M_{2^\infty}^p\otimes^p {A_n}$ has trivial $K$-theory. 
Moreover, since said algebra is a direct limit of $M^p_{2^k}\otimes^p_\spat A_n$, for $k\in\NN$, it is enough to show that $M_{2^k}^p\otimes^p_\spat {A_n}$ has trivial $K$-theory.
Now, $M^p_{2^k}\otimes^p {A_n}$ is isomorphic to ${A_n}$ by \autoref{prop:OnpMatrices}, so the 
claim follows from the inductive step.

Finally, the $K$-theory of $F^p(\ZZ,B\otimes^p {A_n},\beta\otimes\id_{A_n})$ can be computed using the $L^p$-analog of Pimsner--Voiculescu's 6-term exact sequence (Theorem~6.15 in~\cite{Phi13arX:LpCrProd}), which yields
\[
K_0(F^p(\ZZ,B\otimes^p {A_n},\beta\otimes\id_{A_n}))
\cong K_1(F^p(\ZZ,B\otimes^p{A_n},\beta\otimes\id_{A_n}))\cong\{0\},
\]
as desired. 
The last assertion in the statement is \autoref{prp:O2O2Noniso}.
\end{proof}

%==========================================================================================
As remarked in the introduction, this result stands in stark contrast with the Kirchberg--Phillips classification of simple, purely infinite, amenable
\ca s by $K$-theory.
The following is inspired by Kirchberg's 
$\mathcal{O}_2$-embedding theorem \cite{KirPhi00EmbExactCAlg}.

%==========================================================================================
\begin{qst}
Let $p\in [1,\infty)\setminus\{2\}$.
Does every simple, separable, unital, amenable $L^p$-operator algebra 
embed unitally and contractively into $\mathcal{O}_2^p$?
\end{qst}
 
%==========================================================================================
We suspect that this question has a negative answer, and that $\mathcal{O}_2^p\otimes^p\Otp$ is a counterexample. 
However, the techniques developed in this paper seem to be insufficient to rule out the existence of a unital, contractive map $\Otp\otimes^p\Otp\to \Otp$. 
This question is explored in \cite{GarGun24pre:EmbLpOpAlg}.

Another interesting
Banach-algebra completion of the Leavitt path algebra $L_2$ 
has been constructed in \cite{DawHor22PrlyInfCuLike}; 
see Section~3 there and specifically Remark~3.11. 
In the light of \autoref{prp:O2O2Noniso}, it is a natural problem to determine whether the 
Banach algebra constructed there is self-absorbing in a 
suitable sense.

%\bibliographystyle{../../aomalphaMyShort}
%\bibliography{../../References}

\end{document}